\documentclass[10pt,oneside,reqno]{article}
\usepackage[top=1in, bottom=1in, left=1in, right=1in]{geometry}
\usepackage[title,titletoc,toc]{appendix}
\usepackage{amsmath, amsthm, amssymb,subeqnar}
\usepackage{graphicx}
\usepackage{overpic}
\usepackage{cancel}
\usepackage{url}
\usepackage{enumitem}
\usepackage{caption}
\usepackage{color}
\usepackage{float}
\usepackage{rotating}
\usepackage{wrapfig}
\usepackage{multirow}
\usepackage{adjustbox}
\usepackage{setspace}
\usepackage{palatino} 
\usepackage{algorithm,algcompatible}
\usepackage{algpseudocode}
\usepackage{tikz}
\usepackage{booktabs}

\usepackage{mathtools}
\usepackage{psfrag}
\usepackage{tikz}
\usepackage[T1]{fontenc}
\usepackage{array}
\usepackage{makecell}
\newcolumntype{x}[1]{>{\centering\arraybackslash}p{#1}}

\usepackage[bottom,flushmargin,hang,multiple]{footmisc}
\usepackage{lipsum}

\DeclareMathOperator{\sinc}{sinc}
\algnewcommand\algorithmicinput{\textbf{Input:}}
\algnewcommand\INPUT{\item[\algorithmicinput]}
\algnewcommand\algorithmicoutput{\textbf{Output:}}
\algnewcommand\OUTPUT{\item[\algorithmicoutput]}
\algnewcommand\algorithmicoptional{\textbf{Optional:}}
\algnewcommand\OPTIONAL{\item[\algorithmicoptional]}

\newcommand{\bb}{\mathbf{b}}

\newcommand{\bF}{\mathbf{F}}

\newcommand{\bg}{\mathbf{g}}
\newcommand{\bH}{\mathbf{H}}

\newcommand{\bK}{\mathbf{K}}
\newcommand{\bP}{\mathbf{P}}

\newcommand{\bT}{\mathbf{T}}

\newcommand{\bu}{\mathbf{u}}

\newcommand{\bv}{\mathbf{v}}

\newcommand{\bw}{\mathbf{w}}

\newcommand{\bX}{\mathbf{X}}
\newcommand{\bx}{\mathbf{x}}

\newcommand{\bPhi}{\boldsymbol{\Phi}}

\newcommand{\bSigma}{\boldsymbol{\Sigma}}

\definecolor{blue}{rgb}{0,0,1}
\definecolor{darkgreen}{rgb}{0,0.5,0}
\definecolor{red}{rgb}{1,0,0}
\definecolor{teal}{rgb}{0,0.5,0.7}

\newtheorem{remark}{Remark}
\newtheorem{theorem}{Theorem}

\newtheorem{corollary}{Corollary}

\graphicspath{{Figures/}}
\newcommand{\boundellipse}[3]
{(#1) ellipse (#2 and #3)
}
\title{\Large\textbf{Time-Delay Observables for Koopman: Theory and Applications}}
\author{Mason Kamb$^{\dag}$, Eurika Kaiser$^{*}$, Steven L. Brunton$^{*}$, J. Nathan Kutz$^{\dag}$ \\~\\
\small{$^\dag$ Department of Applied Mathematics, University of Washington, Seattle, WA 98195,
United States}\\
\small{$^*$ Department of Mechanical Engineering, University of Washington, Seattle, WA 98195, United States}}

\begin{document}
\maketitle



\begin{abstract}

Nonlinear dynamical systems are ubiquitous in science and engineering, yet analysis and prediction of these systems remains a challenge. Koopman operator theory circumvents some of these issues by considering the dynamics in the space of observable functions on the state, in which the dynamics are intrinsically linear and thus amenable to standard techniques from numerical analysis and linear algebra. However, practical issues remain with this approach, as the space of observables is infinite-dimensional and selecting a subspace of functions in which to accurately represent the system is a nontrivial task. In this work we consider time-delay observables to represent nonlinear dynamics in the Koopman operator framework. We prove the surprising result that Koopman operators for different systems admit universal (system-independent) representations in these coordinates, and give analytic expressions for these representations. In addition, we show that for certain systems a restricted class of these observables form an optimal finite-dimensional basis for representing the Koopman operator, and that the analytic representation of the Koopman operator in these coordinates coincides with results computed by the dynamic mode decomposition. We provide numerical examples to complement our results. In addition to being theoretically interesting, these results have implications for a number of linearization algorithms for dynamical systems.

\end{abstract}

\section{Introduction}
Dynamical systems are ubiquitous in science and engineering.   While linear systems are well-characterized, understanding nonlinear systems remains a open challenge. Nonlinear systems do not satisfy the linear superposition principle, can exhibit an extremely wide range of behaviors, including chaos, and do not generally admit analytic solutions.  Koopman operator theory is an emerging framework for analyzing such systems. In this framework, the finite-dimensional nonlinear state space dynamics are transformed to an infinite-dimensional linear dynamical system in a Hilbert space of functions on the state, which is encoded in the Koopman operator \cite{koopman1931,koopmanvonneumann1932,Mezic2005nd}. The eigenfunctions of the Koopman operator provide a set of coordinates in which the dynamics of the system appear globally linear. While this framework is theoretically powerful, finding these eigenfunctions is a challenging problem which lacks principled approaches.
 The current leading algorithm for computing these eigenfunctions, the {\em dynamic mode decomposition} (DMD) \cite{Schmid2010jfm, rowleymezicetal2009} and its extension to nonlinear observables, using either judiciously selected variables~\cite{PDEKoopman_kutzetal2018complexity} or the {\em extended} DMD (EDMD) algorithm \cite{williamskevrekidisrowley2015}, require the a priori selection of a Koopman-invariant subspace in order to work effectively. If the observables are poorly selected, the resultant approximation to the Koopman operator and its eigenfunctions can be quite poor. In this work, we provide a general and computationally efficient approach to obtaining effective observable functions, which is based on a variant of delay embedding.

Delay embedding is a classical approach to augmenting the information contained in the system state by augmenting it with measurements of the state history. Takens' seminal embedding theorem establishes that under certain technical conditions, delay embedding a signal coordinate of the system can reconstruct the attractor of the original system, up to a diffeomorphism \cite{takens1981}. Delay embedding methods have also been employed for system identification, most notably by the eigensystem realization algorithm (ERA) \cite{juangpappa1985}. An additional variant to delay embedding was introduced by Broomhead and King in \cite{broomheadking1986}, which projects the delay embedded measurements onto the eigenvectors of the autocorrelation function. As these components tend to yield more information and are more robust to noise than generic embedding analysis, this technique has since been widely adopted in a variety of fields, notably as singular spectrum analysis (SSA) \cite{vautardghil1989}.

More recently, delay embedding has shown promise as a technique for computing Koopman eigenfunctions from data, both in regimes where only partial state information is available \cite{bruntonproctorkaiserkutz2017}, as well as in regimes where full state information is available but more functions are needed to span a Koopman-invariant subspace \cite{leclainchevega2017a}. Brunton et. al.~\cite{bruntonproctorkaiserkutz2017} developed a variant of this technique, called the {\em Hankel Alternative View of Koopman} (HAVOK) analysis, which studied the linear dynamics of the projections of time-delay embedded systems onto the singular vectors of a Hankel matrix of the signal. They found the striking empirical result that even for systems where the assumption of Koopman-invariance does not hold, such as chaotic systems, the dynamics of these coordinates was nearly linear, and could be closed by including the action of a single, high-order forcing term. Mezic and Arbabi~\cite{arbabimezic2017} later proved the convergence of these methods for ergodic systems under the assumption that the time-delay subspace was Koopman-invariant.

In the present work, we consider a generalization of HAVOK wherein the dynamics of the system are embedded in {\em convolutional coordinates}. These coordinates are given by the projections of time-delay coordinates onto a generic, infinite, orthonormal basis. We prove the striking result that the dynamics of these coordinates are linear and, for a given choice of basis, independent of the underlying system. The proof is straightforward, and to our knowledge has not been reported in the literature. Although this result is exact in the infinite-dimensional case, finite-dimensional truncations of these dynamics are generally poor at approximating the dynamics of the underlying system. We instead advocate for the use of a basis computed via a Singular Value Decomposition (SVD) of the trajectory data. We show that these coordinates have the following advantageous properties:
\begin{enumerate}
    \item  The analytically computed linear dynamics on the SVD basis match those estimated via a DMD-type algorithm.
    \item For a nonlinear system that admits a Koopman mode expansion of order $N$, the dynamics of the first $N$ convolutional coordinates exactly encode the dynamics of the system, and the associated Koopman eigenfunctions will be in the span of these coordinates. We know of no other family of observable functions that has a similar guarantee.
    \item The SVD basis is also the basis of eigenvectors of the autocorrelation function of the signal. The intrinsic dynamical properties of these convolutional coordinates enable a dramatically faster DMD computation for this set of observables.
\end{enumerate}
We prove and conjecture some results relating to the nature and quality of these approximations to the Koopman operator in a general setting. We compute the discrete approximations of these coordinates for a number of example systems and the associated Koopman spectrum and eigenfunctions. We observe that our methods generally work well and provide superior or comparable Koopman approximations to other families of observables. A few pathological cases are identified for which our method does not perform as expected. In these cases, the eigenvalues of the system are nearly degenerate and the delay embedding window is small. 

\section{Background}
\subsection{Koopman Theory}
We consider an autonomous dynamical system on $\mathbb{R}^n$ of the form
\begin{equation}\label{eq:defnonlin}
    \frac{d}{dt} \mathbf{x}(t) = \mathbf{f}(\mathbf{x}(t))
\end{equation}
where $\mathbf{x} \in \mathbb{R}^n$ represents the state of the system and $\mathbf{f}:\mathbb{R}^n\rightarrow\mathbb{R}^n$ is a possibly nonlinear vector field. Here, $t$ denotes time. For all applications in this paper, we will assume that $\mathbf{f}$ is at least continuous so that the trajectories $\mathbf{x}(t)$ are at least $\mathcal{C}^1$. The system (\ref{eq:defnonlin}) induces a flow map $\bF_t$, specifically a time-dependent family of maps $\{ \bF_t\}_{t\in\mathbb{R}_{+}}$, satisfying
\begin{equation}
    \mathbf{x}(t) = \bF_t(\mathbf{x}_0),
\end{equation}
which yields the state $\mathbf{x}(t)$ at time $t$ given an initial state ${\bf x}_0:={\bf x}(0)$.

B.~O. Koopman and J.~v. Neumann introduced an influential operator-theoretic perspective on dynamical systems of the above form \cite{koopman1931,koopmanvonneumann1932}. Their fundamental insight was that the finite-dimensional nonlinear dynamics of (\ref{eq:defnonlin}) can be transformed to an infinite-dimensional, linear dynamical system by considering an appropriately chosen Hilbert space of scalar observable functions $g:\mathbb{R}^n \to \mathbb{C}$ on the state, instead of the state directly. 
The one-parameter semigroup of \textit{Koopman Operators}, $\{\mathcal{K}_t\}_{t\in\mathbb{R}_{+}}$, acting on observables $g$ is defined by
\begin{equation}\label{eq:defkoopmanflow}
    \mathcal{K}_t g = g \circ \bF_t,
\end{equation}
where $\circ$ denotes the composition operator so that
\begin{equation}
	\mathcal{K}_t g({\bf x}_0) = g(\bF_t({\bf x}_0)) = g({\bf x}(t)).
\end{equation}
$\mathcal{K}_t$ maps the measurement function $g$ to the values it will take after the dynamics have progressed for a time $t$. The generator of the family of Koopman operators is defined by 

\begin{equation}\label{eq:defkoopmangen}
    \mathcal{K} g = \frac{d}{dt} g \circ \textbf{F}_t = \lim_{t \to 0^+} \frac{1}{t}(\mathcal{K}_t g - g)
\end{equation}
The second equality holds in an $L^2$ sense when $\textbf{f}$ is smooth and globally Lipshitz \cite{Lasota1994book,Mezic2017book}. For a system that is not globally Lipshitz, one may consider a reduced compact domain $\Omega$; for $L^2$ functions with compact support on this domain, the second equation holds in an $L^2$ sense, even if the dynamics are not globally Lipshitz. More generally, for a continuous observable function, this limit will hold pointwise, rather than in $L^2$ sense.

These operators are linear by construction and reflect the dynamics of the system in the space of observable functions. This framework allows us to transform the possibly complex, nonlinear dynamics in the finite-dimensional state space for infinite-dimensional linear dynamics in the space of observable functions on the state. While this framework is theoretically appealing, it is difficult to work directly with the infinite-dimensional dynamics of (\ref{eq:defkoopmanflow}). An alternative approach is to perform a spectral decomposition of the Koopman Operator \cite{Mezic2005nd}. 
This approach seeks the eigenvalues $\lambda\in\mathbb{C}$ and eigenfunctions $\varphi:\mathbb{R}^n\rightarrow\mathbb{C}$ of the Koopman operator, satisfying:
\begin{equation}\label{eq:koopmaneigen}
	\mathcal{K}_t \varphi(\bx_0) = \varphi(\bF_t(\bx_0)) = \varphi(\bx_0) e^{\lambda t}.
\end{equation}
Each eigenfunction can be considered as a special type of observable, which behaves linearly in time.
Thus, the eigenspace of each eigenfunction $\varphi$ is Koopman invariant by construction, so these functions provide coordinates in which the Koopman operator is naturally finite-dimensional. Another use of the Koopman eigenfunctions is as a basis for the space of observables. If the Koopman operator family $\mathcal{K}_t$ has a pure point spectrum, we can expand the dynamics of any vector-valued observable function $\mathbf{g}:\mathbb{R}^n\to\mathbb{C}^d$ in a basis of these eigenfunctions:
\begin{equation}\label{eq:KMD}
    \mathbf{g}(\mathbf{x}(t)) = \mathcal{K}_t\mathbf{g}(\mathbf{x}_0) = \sum_{j=0}^{\infty} \mathbf{c}_j \varphi_j(\mathbf{x}_0) e^{\lambda_j t},
\end{equation}
In the special case where the Koopman operator is unitary, the eigenfunctions $\varphi_j$ are orthogonal, and these coefficients satisfy $\textbf{c}_j = \langle \varphi_j, \textbf{g} \rangle$. This decomposition is known as the {\em Koopman mode decomposition} \cite{Mezic2005nd}. This decomposition illustrates a powerful application of Koopman eigenfunctions. If a complete basis of these eigenfunctions exists, then the dynamics of any set of observables can be solved exactly and linearly using the expansion above. 

More generally, for systems with an invariant measure $\mu$, the action of the Koopman operator on the Hilbert space $L^2(\mu)$ is unitary, and can be decomposed into the following form:
\begin{equation}
    \mathcal{K}_t \textbf{g} = \sum_{j = 0}^\infty \textbf{c}_j \phi_j(\textbf{x}_0) e^{2\pi i \omega_j t} + \int dE(\omega) e^{2\pi i \omega} 
\end{equation}
where $dE(\omega)$ is the spectral measure associated with the continuous component of the spectrum of $\mathcal{K}_t$.

\subsection{Dynamic Mode Decomposition}
While the Koopman operator-theoretic perspective on dynamical systems dates back to the 1930s, it has only been in recent years that these methods have gained widespread attention. One reason for this is the variety of computational tools and algorithms that have been developed to aid the study of the Koopman operator and its spectral properties \cite{Mezic2005nd,Mezic2005nd,Budivsic2012chaos,Lan2013physd,Mezic2017book}. One such class of methods is known as dynamic mode decomposition (DMD)~\cite{Schmid2010jfm,rowleymezicetal2009,Tu2014jcd,Kutz2016book}. DMD was originally introduced in the fluid dynamics community as a tool for extracting spatiotemporal coherent structures from complex flows \cite{Schmid2010jfm}. It was later shown that the spatiotemporal modes obtained by DMD converge to the Koopman modes for a set of linear observables \cite{rowleymezicetal2009,Tu2014jcd}. This property of the DMD has made it one of the primary computational tools in Koopman Theory.

The exact DMD algorithm works as follows~\cite{Tu2014jcd}. A snapshot matrix of state measurements $\mathbf{X}$ is assembled. The algorithm estimates a linear relationship between the data matrices
\begin{align*}
    \mathbf{X}_1 &= \begin{pmatrix}
    \vert&\vert&&\vert\\
    \mathbf{x}_1 & \mathbf{x}_2 & \cdots & \mathbf{x}_{M-1}\\
    \vert&\vert&&\vert
    \end{pmatrix}
    & \mathbf{X}_2 = \begin{pmatrix}
    \vert&\vert&&\vert\\
    \mathbf{x}_2 & \mathbf{x}_3 & \cdots & \mathbf{x}_M\\
    \vert&\vert&&\vert
    \end{pmatrix}.
\end{align*}
where $\mathbf{x}_j = \mathbf{x}(t_j)$ denotes the $j$th timestep at discrete time $t_j$ and $\Delta t=t_{j+1}-t_j$ is the timestep between snapshots. The algorithm estimates the propagator matrix
$\mathbf{K} \in \mathbb{R}^{n\times n}$ that satisfies
\begin{equation}
	\bX_2 \approx \bK \bX_1
\end{equation}
The optimal $\mathbf{K}$ is found by solving the optimization problem
\begin{equation}
    \mathbf{K} = \arg\min\limits_{\hat{\mathbf{K}}} \Vert \hat{\mathbf{K}}\mathbf{X}_1 - \mathbf{X}_2 \Vert_F
\end{equation}
where $|| \cdot||_F$ denotes the Frobenius norm defined as
    $\Vert \mathbf{X} \Vert_F = \sqrt{\sum_{j = 1}^n\sum_{k = 1}^M X_{jk}^2}$.
The least-squares solution to this optimization problem is known to be
   $\mathbf{K} = \mathbf{X}_2 \mathbf{X}_1^{\dagger}$
where $\dagger$ denotes the Moore-Penrose pseudoinverse. Once an estimate for $\mathbf{K}$ has been produced through this procedure, we can compute the eigendecomposition of $\mathbf{K}$:
\begin{equation}
    \mathbf{K} = \mathbf{\Phi}\pmb{\Lambda}\mathbf{\Phi}^{-1}.
\end{equation}
The columns $\pmb{\phi}_j\in\mathbb{R}^n$ of $\mathbf{\Phi}$ are called the DMD {\em modes} of $\mathbf{K}$. While this optimization problem is typically solved using a naive least-squares solution, there exist alternative algorithms which produce a less biased set of DMD modes and eigenvalues.  For details,  see Askham and Kutz~\cite{askham2018variable} for a comparison of existing techniques.

Using this eigendecomposition, we can predict the future solution approximately in the following form:
\begin{equation}
\bx_n \approx \sum_{j=1}^N b_j \pmb{\phi}_j \lambda_j^n
\end{equation}
where $b_j$'s are the coordinates of the  initial state $\bx_1$ in the Koopman mode basis so that $\bx_1 = \bPhi\bb$, i.e.\ initial amplitude of each mode $\pmb{\phi}_j$.
In continuous-time, we obtain:
\begin{equation}
    \bx(t) \approx \sum_{j=1}^N b_j \pmb{\phi}_j e^{\omega_j t}
\end{equation}
with the continuous-time eigenvalue $\omega_j=\ln(\lambda_j)/\Delta t$.
Here, $\pmb{\phi}_j$ and $\lambda_j$ are the $j$th eigenvector-eigenvalue pair of $\mathbf{K}$. 


In practice, this eigendecomposition may be prohibitively expensive to compute if the dimension of $\mathbf{K}$ is large. This issue can be circumvented by computing the {\em singular value decomposition} (SVD) of $\mathbf{X}_1 = \mathbf{U}\mathbf{\Sigma}\mathbf{V}^*$ and computing the $r\times r$ truncated matrix $\tilde{K}$:
\begin{equation}
\tilde{\mathbf{K}} = \mathbf{U}_r \mathbf{K} \mathbf{U}_r^*
\end{equation}
where $\mathbf{U}_r$ consists of the first $r$ columns of $\mathbf{U}$. We then compute the eigendecomposition of $\tilde{\mathbf{K}}$
\begin{equation}
    \tilde{\mathbf{K}} =\tilde{\pmb{\Phi}}\tilde{\mathbf{\Lambda}}\tilde{\pmb{\Phi}}^{-1}.
\end{equation}
The full-state eigenvector matrix $\pmb{\Phi}$ can then be estimated as follows:
\begin{equation}
    \mathbf{\Phi} = \mathbf{X}_2 \mathbf{V}\mathbf{\Sigma}^{-1}\tilde{\pmb{\Phi}}.
\end{equation}
Related DMD algorithms and innovations include the optimized DMD~\cite{askham2018variable}, Bayesian DMD~\cite{Takeishi2017JCAI}, and subspace DMD~\cite{Takeishi2017PhysRevE}.   Optimized DMD exhibits particularly stable numerical properties~\cite{askham2018variable}.  


In many real-world systems, the underlying system dynamics will not be linear, and thus the linear operator $\mathbf{K}$ estimated by DMD will not provide a good approximation to the nonlinear system dynamics. We may instead decide to work with an augmented state consisting of possibly nonlinear functions of the state $\mathbf{g} = (g_1, \cdots g_D)$, where $D$ is generally larger than the dimension of the original state $N$. We can then compute the DMD on the new data matrices
\begin{align*}
    \mathbf{g}(\mathbf{X}_1) &= \begin{pmatrix}
    \vert & \vert & & \vert\\
    \mathbf{g}(\mathbf{x}_1)&\mathbf{g}(\mathbf{x}_2) & \cdots & \mathbf{g}(\mathbf{x}_{M - 1})\\
    \vert & \vert & & \vert
    \end{pmatrix}
    &
    \mathbf{g}(\mathbf{X}_2) = \begin{pmatrix}
    \vert & \vert & & \vert\\
    \mathbf{g}(\mathbf{x}_2)&\mathbf{g}(\mathbf{x}_3) & \cdots & \mathbf{g}(\mathbf{x}_{M})\\
    \vert & \vert & & \vert
    \end{pmatrix}.
\end{align*}
This allows us to estimate a nonlinear model of the form $\mathbf{g}^{-1} \circ \mathbf{K} \circ \mathbf{g}$ on the state $\mathbf{x}$. This procedure is known as the \textit{extended dynamic mode decomposition} (EDMD) \cite{williamskevrekidisrowley2015} or the variational approach of conformation dynamics (VAC)~\cite{noe2013variational,nuske2014jctc}. The classic DMD is a special case of EDMD in the case where the measurement consists of the identity $\mathbf{g}(\mathbf{x}) = \mathbf{x}$.

A theoretical benefit of EDMD is that the estimated linear operators are known to converge to an orthogonal projection of the Koopman operator onto the subspace of observables spanned by $\mathbf{g}$ in the limit of infinite data \cite{kordamezic2017JNS}. Furthermore, in the limit as the observables in $\mathbf{g}$ span the relevant Hilbert space of observable functions, the action of the Koopman operator on this space can be exactly reconstructed \cite{kordamezic2017JNS}. 
However, EDMD is limited by the quality of the selection of observable functions, since a generic selection of a finite number of observables $\mathbf{g}$ will not span a Koopman-invariant subspace~\cite{Brunton2016plosone} and thus the orthogonal projection may discard the relevant dynamics of the system.  
Augmenting $\mathbf{g}$ with additional measurements increases the span of these observables and thus may result in a higher-quality approximation. 
However, this can result in an exceptionally large state, making EDMD prone to overfitting.   
The VAC approach explicitly cross-validates the resulting model to avoid overfitting~\cite{noe2013variational,nuske2014jctc}.  
In some cases, machine learning methods \cite{lidietrichbolltkevrekidis2017chaos} or domain knowledge \cite{PDEKoopman_kutzetal2018complexity} can result in a more effective observable selection. For example, Neural networks and deep learning have recently been used to great advantage to represent the Koopman operator~\cite{Takeishi2017nips,Yeung2019ACC,Wehmeyer2018JCP,Mardt2017arxiv,Lusch2018NatCom,OttoRowley2019SIADS}. However, these techniques often lack formal justification.  

\subsection{Time Delay Embedding}
A classic technique for dealing with partial state information is \textit{delay embedding}. This method augments a state represented by a single (or few) measurement function with its past history, resulting in a new observable $\tilde{\bg}(\bx(t)) =: (g(\mathbf{x}(t)), g(\mathbf{x}(t - \Delta t)), g(\mathbf{x}(t - 2\Delta t)), \cdots, g(\mathbf{x}(t - n\Delta t)), \cdots, g(\mathbf{x}(t - (N-1)\Delta t))) \in \mathbb{R}^{N}$ with the lag time $\Delta t$. This is known as a \textit{delay embedding} of the trajectory $g(\mathbf{x}(t))$. It was shown by Takens \cite{takens1981} that under mild conditions on the observable $g$, the dynamics of the delay vector $\mathbf{y}$ are guaranteed to be diffeomorphic to the dynamics of the original state $\mathbf{x}$, provided that the embedding dimension is $N \geq 2n + 1$, where $n$ is the dimension of the state. In many cases, a smaller embedding dimension may be chosen without sacrificing the diffeomorphism.

In practice, the quality of the reconstruction may depend greatly on parameter choices, such as the lag time $\delta$ and the embedding dimension $N$ \cite{liebertschuster1988,Gibson1992phD}. One approach for dealing with these issues is to compute the principal components of the trajectory \cite{broomheadking1986,vautardghil1989}. These are obtained by computing the SVD of a Hankel matrix on the trajectory data:
\begin{align}\label{eq:hankelsvd}
    \mathbf{H} &= \begin{pmatrix}
    g(\mathbf{x}_1) & g(\mathbf{x}_2)&\cdots & g(\mathbf{x}_M) \\
    g(\mathbf{x}_2) & g(\mathbf{x}_3) & \cdots & g(\mathbf{x}_{M+1})\\
    \vdots & \vdots & \ddots & \vdots\\
    g(\mathbf{x}_N) & g(\mathbf{x}_{N+1}) & \cdots & g(\mathbf{x}_{N+M-1})
    \end{pmatrix}
    =
    \begin{pmatrix}
    \vert & \vert & & \vert\\
    \tilde{\bg}(\bx_1)&\tilde{\bg}(\bx_2)&\cdots&\tilde{\bg}(\bx_M)\\
    \vert & \vert & & \vert
    \end{pmatrix}
= \mathbf{U}\pmb{\Sigma}\mathbf{V}^*.
\end{align}
The principal components (i.e. the right singular vectors of the SVD decomposition) of the trajectory are the columns of $\mathbf{V}$. Each row $\mathbf{v}_j$ are the embeddings of the original states $\mathbf{x}(t)$. An appropriate dimension for the embedded state can be chosen by examining the singular value of each component and truncating after these pass below an appropriate threshold.  Gavish and Donoho, for instance, provide a principled way to apply such thresholding~\cite{gavish2014optimal}.

This approach is advantageous because each principal component ${\bf v}_j$ is normalized and uncorrelated with the other components of the trajectory. These ensure that the reconstructed attractor will not be {\em stretched} along a particular axis, and that each component carries as little redundancy with the others as possible. Furthermore, the subspace spanned by the first $r$ principal components is the best rank-$r$ approximation of the trajectory space in a least-squares sense, so the principal component basis is in this sense an optimal basis for state-space reconstruction. We will return to these considerations at a later point.

The Hankel/delay embedded representation of the state trajectory has been recently connected to Koopman theory~\cite{bruntonproctorkaiserkutz2017}. The Hankel matrix can be rewritten using the action of the Koopman operator on $g$:
\begin{equation}
    \mathbf{H} = \begin{pmatrix}
    g(\bx_1) & \mathcal{K}_{\Delta t} g(\bx_1) & \cdots & \mathcal{K}_{\Delta t}^{M-1} g(\bx_1)\\
    \mathcal{K}_{\Delta t} g(\bx_1) & \mathcal{K}_{\Delta t}^2 g(\bx_1) & \cdots & \mathcal{K}_{\Delta t}^M g(\bx_1)\\
    \vdots & \vdots & \ddots & \vdots\\
    \mathcal{K}_{\Delta t}^{N-1} g(\bx_1) & \mathcal{K}_{\Delta t}^{N} g(\bx_1) & \cdots & \mathcal{K}_{\Delta t}^{M+N-2} g(\bx_1)
    \end{pmatrix}
    = \begin{pmatrix}
    \vert & \vert & & \vert\\
    \tilde{\bg}(\bx_1) & \mathcal{K}_{\Delta t}\tilde{\bg}(\bx_1) & \cdots & \mathcal{K}_{\Delta t}^{M-1}\tilde{\bg}(\bx_1)\\
    \vert & \vert & & \vert
    \end{pmatrix}.
\end{equation}
The delay-embedded state vector can be viewed as the vector of observables, $\tilde{\bg} = (g, K_{\Delta t} g, K^2_{\Delta t} g, \cdots, K^N_{\Delta t} g)$. Due to the special connection with time delay observables and the Koopman operator, they are a natural basis for DMD and related Koopman spectral techniques. 
Delay embeddings have been used previously in DMD analyses in cases where only partial state information is available \cite{bruntonjohnsonojemannkutz2016,Tu2014jcd}.  
Performing DMD on delay coordinates for linear systems is closely related to the eigensystem realization algorithm (ERA)~\cite{juangpappa1985} and singular spectrum analysis (SSA)~\cite{vautardghil1989}.
A nonlinear variant of SSA based on Laplacian spectral analysis has been useful for time series with intermittent phenomena~\cite{Giannakis2012pnas}.  

Brunton et. al. formally introduced this approach of performing a sparsified linear and nonlinear model regression on delay coordinates~\cite{Brunton2016pnas} and established the connection with Koopman theory and chaotic systems in \cite{bruntonproctorkaiserkutz2017}; the approach is referred to as the Hankel Alternative View of Koopman (HAVOK) analysis. 
They found that these coordinates provided nearly Koopman-invariant subspaces, as well as exhibiting several other interesting properties. 
Arbabi and Mezic later studied the properties of HAVOK models \cite{arbabimezic2017}. They establish that for ergodic systems, HAVOK converges to the true Koopman eigenfunctions and eigenvalues of the system. This convergence result is unusual for generic families of measurement functions, and motivates the application of Koopman spectral methods to delay coordinates for systems where incomplete measurement data is available. HAVOK models are also closely related to the Prony approximation of the Koopman decomposition~\cite{Susuki2015cdc}.  

Das and Giannakis~\cite{DasGiannakis2019} have also investigated the spectrum of the Koopman operator in delay coordinates. They provided a kernel integral operator method for approximating the Koopman operator from observable data, using delay coordinates. They showed that, in the limit of infinitely many delay coordinates, their approximation converges to the Koopman operator on the subspace spanned by Koopman eigenfunctions. They also catalog some interesting empirical observations concerning the bi-diagonal structure of the resultant Koopman approximations, which was originally noticed in \cite{bruntonproctorkaiserkutz2017}. Giannakis also investigated the properties of Koopman operators of ergodic systems in delay coordinates ~\cite{Giannakis2019AppHarm}. This paper established a correspondence between Laplace-Beltrami operators on delay-coordinate spaces and the Koopman operator. Our work is motivated in part by the theoretical successes of these works, and in part by the observation of a bi-diagonal structure, which we believe the results presented in this paper can explain.

Delay embedding can also be used to augment the state vector even when complete measurement data is available. This may be desirable for nonlinear systems with broad-spectrum behavior. The DMD can extract at most $N$ dynamic modes and eigenvalues, where $N$ is the dimension of vector $\mathbf{y}$. This is often insufficient for nonlinear systems, wherein a larger range of frequencies can be active even for a simple state. Le Clainche and Vega introduced \textit{higher order dynamic mode decomposition} (HODMD) to circumvent these issues \cite{leclainchevega2017a}. HODMD computes DMD on a state where the delay embedding is across multiple dimensions. The resultant state has dimension $nN$, where $N$ is the length of the delay embedding and $n$ is the dimension of the underlying state. They found that this approach was able to find a larger range of frequencies and produce more accurate reconstructions of the dynamics for a large variety of nonlinear systems.

\section{Convolutional Coordinates}

A common idea in signal processing is to extract features from signals by convolving with a filter function. For example, convolving signals with Gaussians is known to remove noise from the signal, while convolving the signal with a wavelet basis can be used to extract interpretable feature set for use in classification. Indeed, signal processing analysis is dominated by filtering of time series data. The fundamental insight of these approaches is that looking at local segments of the trajectory or signal offers more information than looking only at a single point. In addition to this, convolutions can also be thought of as linear coordinates on the time-delay embedded state of the trajectory. In this section, we develop a formalism for interpreting these coordinates as functions directly on the state space, which we call \textit{convolutional coordinates}. We prove that for an appropriately chosen set of such coordinates, the dynamics of the system are intrinsically linear. Furthermore, the representation of these dynamics depends only on the choice of convolution functions, and not on the intrinsic dynamics of the system. We connect these results to Koopman theory and argue that these coordinates provide \textit{system-independent representations} of the Koopman operator. These representations are intrinsically infinite-dimensional, and in general the maps obtained by projecting orthogonally onto do not coincide with the closest approximation of the Koopman on these coordinates. This makes these results interesting but not always suitable for practical usage. To circumvent this problem, careful attention must be paid to choose a basis for which these two maps coincide. One basis for which this holds are the SVD basis vectors for the trajectory, which are discussed in section 4.

We consider the dynamical system defined in~\eqref{eq:defnonlin} with ${\mathbf x}\in\mathbb{R}^n$ and an observable on this state ${\mathbf g}:\mathbb{R}^n \to \mathbb{R}^k$ (in many applications of delay embedding, this observable is in fact a scalar (i.e. $k = 1$), and ${\mathbf g}(\textbf{x})$ will select a single component of the state vector, e.g. ${\mathbf g}(\mathbf{x}) = x_1$). Let $\pmb{\phi}_j:[-\tau,\tau] \to \mathbb{R}^k$ be a differentiable, orthonormal basis on the interval $[-\tau,\tau]$, with the inner product $\langle \pmb{\phi}_j, \pmb{\phi}_k \rangle = \int_{-\tau}^\tau \pmb{\phi}_j(s) \cdot \pmb{\phi}_k(s)\,ds$.

We define the \textit{convolutional coordinates} $w_j$ associated with this basis set as follows:
\begin{equation}\label{eq:defconvcoords}
	w_j(\mathbf{x}) 
	= \langle \pmb{\phi}_j, \mathcal{K}_{(\cdot)} \textbf{g}(\textbf{x}) \rangle 
	= \int_{-\tau}^\tau \pmb{\phi}_j(s) \cdot \mathcal{K}_s \mathbf{g}(\textbf{x})\,ds,
\end{equation}
For a specific trajectory $\textbf{x}(t)$, the evolution of the convolutional coordinate can be written in the intuitive form
\begin{equation}
	w_j(\mathbf{x}(t)) 	= \int_{-\tau}^\tau \pmb{\phi}_j(s) \cdot {\mathbf g}(\mathbf{x}(t+s))\,ds .
\end{equation}
This corresponds to convolving the basis functions $\pmb{\phi}_j$ with the trajectory ${\mathbf g}(\mathbf{x}(t))$ over the time window $[-\tau,\tau]$. Since the functions $\pmb{\phi}_j$ are orthonormal, the convolutional coordinates can also by thought of as the projections of these basis functions onto local segments of the measurement trajectory. We can reconstruct $\mathcal{K}_s \textbf{g}(\textbf{x})$ using a basis function expansion:
\begin{equation}\label{eq:slidingwindowreconstruction}
\mathcal{K}_s{\mathbf g}(\bx) 
= \sum_{j=0}^{\infty} w_j(\bx) \pmb{\phi}_j(s)
\end{equation}
For a general basis set, this expansion will only hold in an $L^2$ sense on $s \in [-\tau, \tau]$, rather than pointwise; however, for appropriate observables and basis functions (e.g. a continuous observable and a basis of Chebyshev polynomials), this equation will hold pointwise for all $s$.

We are now in a position to demonstrate a remarkable fact about these coordinates: for appropriately chosen basis sets and observable functions, the action of the Koopman operator on these coordinates can be described in simple, analytic form. Even more remarkably, this representation does not depend on the underlying system dynamics.

\begin{theorem}\label{maintheorem}
Let $\mathbf{g}$ be a continuous $L^2$ function on $\mathbb{R}^n$ and let $\{\pmb{\phi}_j \}$ be an orthonormal, differentiable basis set for the set of continuous $L^2$ functions on $[-\tau,\tau]$. Suppose that for some subset $\Omega \subset \mathbb{R}^n$, for each point $\textbf{x} \in \Omega$, for $s \in [-\tau,\tau]$, the basis function expansion $\sum_{j = 0}^\infty w_j(\textbf{x})\pmb{\phi}_j(s)$ converges uniformly (with respect to $s$) to $\mathcal{K}_{s}\mathbf{g}(\textbf{x})$. Then the action of the Koopman generator on $w_j |_{\Omega}$ is given, pointwise on $\Omega$, by

\begin{equation}\label{eq:convdifferentialdynamics}
    \mathcal{K} w_j(\textbf{x}) = \sum_{k=0}^{\infty} K_{jk} w_k(\textbf{x})
\end{equation}
where
\begin{equation}\label{eq:convcoeffs}
    K_{jk} =  \langle \pmb{\phi}_j,\pmb{\phi}_k'\rangle = \int_{-\tau}^\tau \pmb{\phi}_j(s) \cdot \pmb{\phi}_k'(s)\,ds .
\end{equation}

\end{theorem}

\begin{proof}
Consider a point $\textbf{x} \in \Omega$. We then have
\begin{align*}
    \mathcal{K} w_j(\textbf{x}) &= \mathcal{K} \int_{-\tau}^\tau \pmb{\phi}_j(s) \cdot \mathcal{K}_s \mathbf{g}(\textbf{x})\,ds = \frac{d}{dt} \int_{-\tau}^{\tau}\pmb{\phi}_j(s) \cdot \textbf{g} \circ \textbf{F}_{s + t}(\textbf{x})\,ds\, \bigg|_{t = 0} \\ 
\end{align*}
Making the substitution $s' = t + s$ in the above integral yields
\begin{align*}
    \frac{d}{dt}\int_{-\tau}^{\tau}\pmb{\phi}_j(s) \cdot \textbf{g} \circ \textbf{F}_{s + t}(\textbf{x})\,ds\, \bigg|_{t = 0} &= \frac{d}{dt} \int_{t-\tau}^{t + \tau} \pmb{\phi}_j(s' - t) \cdot \mathcal{K}_{s'} \mathbf{g}(\textbf{x})\,ds' \,\bigg|_{t = 0}\\
    &= \bigg[\pmb{\phi}_j(s' - t) \cdot \mathcal{K}_{s'} \mathbf{g}(\textbf{x}) \bigg|_{t - \tau}^{t + \tau} - \int_{t - \tau}^{t + \tau} \pmb{\phi}'_j(s' - t) \cdot \mathcal{K}_{s'} \mathbf{g}(\textbf{x})\, ds'\bigg] \bigg|_{t = 0}\\
    &= \pmb{\phi}_j(s) \cdot \mathcal{K}_{s} \mathbf{g}(\textbf{x}) \bigg|_{-\tau}^{\tau} - \int_{-\tau}^\tau \pmb{\phi}'_j(s) \cdot \mathcal{K}_{s} \mathbf{g}(\textbf{x})\,ds
\end{align*}
where we have evaluated the above derivative using Leibniz' integral rule $\frac{d}{dt} \int_{a(t)}^{b(t)} f(s,t)\,ds = f(b(t),t)b'(t) - f(a(t),t)a'(t) + \int_{a(t)}^{b(t)} \partial_t f(s,t)\,ds$. At this point we may expand $\mathcal{K}_{s'} \mathbf{g}(\textbf{x})$ using (\ref{eq:slidingwindowreconstruction}) and substitute this into the above integral:
\begin{align*}
    \pmb{\phi}_j(s) \cdot \mathcal{K}_s \mathbf{g}(\textbf{x}) \bigg|_{-\tau}^\tau - \int_{-\tau}^\tau \pmb{\phi}_j'(s) \cdot \mathcal{K}_s \mathbf{g}(\textbf{x})\,ds
    &=  \pmb{\phi}_j(s) \cdot \mathcal{K}_s \mathbf{g}(\textbf{x}) \bigg|_{-\tau}^\tau - \int_{-\tau}^\tau \pmb{\phi}_j'(s) \cdot \sum_{k = 0}^\infty w_k(\textbf{x})\pmb{\phi}_k(s)\,ds\\
    &= \pmb{\phi}_j(s) \cdot \mathcal{K}_s \mathbf{g}(\textbf{x}) \bigg|_{-\tau}^\tau - \sum_{k = 0}^\infty w_k(\textbf{x}) \int_{-\tau}^\tau \pmb{\phi}_j'(s) \cdot \pmb{\phi}_k(s)\,ds\\
\end{align*}

Finally, integrating by parts again, we obtain
\begin{align*}
    \mathcal{K} w_j(\textbf{x}) &= \pmb{\phi}_j(s) \cdot \mathcal{K}_s \mathbf{g}(\textbf{x}) \bigg|_{-\tau}^\tau - \sum_{k = 0}^\infty w_k(\textbf{x}) \bigg[ \pmb{\phi}_j(s) \cdot \pmb{\phi}_k(s) \bigg|_{-\tau}^\tau - \int_{-\tau}^\tau \pmb{\phi}_j(s) \cdot \pmb{\phi}_k'(s) \,ds\bigg]\\
    &= \pmb{\phi}_j(s) \cdot \mathcal{K}_s \mathbf{g}(\textbf{x}) \bigg|_{-\tau}^\tau - \pmb{\phi}_j(s) \cdot \sum_{k = 0}^\infty w_k(\textbf{x}) \cdot \pmb{\phi}_k(s) \bigg|_{-\tau}^\tau + \sum_{k = 0}^\infty w_k(\textbf{x}) \langle \pmb{\phi}_j, \pmb{\phi}_k' \rangle\\
    &= \sum_{k = 0}^\infty w_k(\textbf{x}) \langle \pmb{\phi}_j, \pmb{\phi}_k' \rangle
\end{align*}

\end{proof}

{Special care must be taken when choosing a domain $\Omega$ and a set of basis functions $\pmb{\phi}_j$ in order to ensure the validity of the representation (\ref{eq:convdifferentialdynamics}). A naively chosen basis will generally not satisfy the hypotheses of the theorem over the entire domain of the dynamical system. For instance, expanding the a local trajectory segment $g(\textbf{x}(t + s))$ in the Fourier basis $\phi_{n}(s) = \exp(2\pi i n s/\tau)$ will generally only result in a uniformly convergent series when the domain $\Omega$ is a $2\tau$-periodic orbit of the system. Other basis sets offer more flexibility: for instance, one can expand an arbitrary Lipshitz-continuous function in a basis of Chebyshev polynomials, and this expansion is guaranteed to be absolutely uniformly convergent \cite{townsendtrefethen2015}. For such bases, the convolutional coordinates have the further, remarkable property that the choice of window length $\tau$ can be arbitrary (in practice, the choice of $\tau$ may affect reconstruction quality, due to the nature of finite sampling of the trajectory; we refer the reader to \cite{liebertschuster1988,Gibson1992phD} for discussions of practical considerations for the selection of a delay window in the setting of qualitative attractor reconstruction).

Unless almost all trajectories in $\Omega$ remain in $\Omega$, the existence of this representation of the Koopman generator on $\Omega$ does not guarantee that these coordinates form a Koopman-invariant subspace. For systems with an invariant measure, one may choose $\Omega$ to be the support of the invariant measure, in which case, provided the basis function expansions converge appropriately, the convolutional coordinates form a Koopman-invariant subspace. }

An analogous result holds for the discrete-time dynamics of the convolutional coordinates, with the additional assumptions of the analyticity and analytic continuability of both the trajectory and the basis function set.  These are significantly stronger assumptions than those required for theorem \ref{maintheorem}, to the point where this result may not be useful for a large variety of systems; however, we include it for completeness below.

\begin{theorem}\label{thm:analytictheorem}
Let $\mathbf{x}(t)$ be a differentiable function on $\mathbb{R}$. Suppose $\mathbf{x}(t)$ is locally analytic everywhere with radius of convergence at least $\tau + \Delta t$. Let $\{\pmb{\phi}_j\}_{j=0}^{\infty}$ be a differentiable, orthogonal basis of $L^2[-\tau, \tau]$. Suppose also that each $\pmb{\phi}_j$ can be analytically continued to a radius $\tau + \Delta t$. Then the action of the Koopman operator $\mathcal{K}_{\Delta t}$ on the convolutional coordinates $w_j$ is given by
\begin{equation}\label{eq:convdiscretedynamics}
    \mathcal{K}_{\Delta t} w_j 
    = \sum_k w_k \int_{-\tau}^\tau \pmb{\phi}_j(s) \cdot \hat{\pmb{\phi}}_k(s + \Delta t)\,ds
    = \sum_k w_k \langle \pmb{\phi}_j(s), \hat{\pmb{\phi}}_k(s + \Delta t)\rangle_s.
\end{equation}
where $\hat{\pmb{\phi}}_k(t)$ denotes the analytic continuation of $\pmb{\phi}_k$ on the interval $[-\tau-\Delta t,\tau+\Delta t]$
\end{theorem}
\begin{proof}
We have
\begin{align*}
    \mathcal{K}_{\Delta t} w_j (\mathbf{x}(t)) &= w_j(\mathbf{x}(t + \Delta t))\\
    &= \int_{-\tau}^\tau \pmb{\phi}_j(s) \cdot \mathbf{g}(\bx(t + \Delta t + s))\rangle\,ds\\
    &= \int_{-\tau}^\tau \pmb{\phi}_j(s) \cdot \sum_{k=0}^{\infty} w_k(\bx(t)) \hat{\pmb{\phi}}_k(s + \Delta t)\,ds\\
    &= \sum_{k=0}^{\infty} w_k(\mathbf{x}(t)) \int_{-\tau}^\tau\pmb{\phi}_j(s) \cdot \hat{\pmb{\phi}}_k(s + \Delta t)\,ds.
\end{align*}
\end{proof}

In summary, the action of the Koopman generator on the convolutional coordinates $\{w_j\}_{j=0}^{\infty}$ admits a concise, analytic representation that is system-invariant, depending only on the choice of basis functions. Provided the system admits an invariant measure, these functions form a Koopman-invariant subspace. 

\subsection{Other Universal Representations}

It may seem surprising that the coefficients in the above expansion are apparently independent of both the underlying dynamical system, as well as the measurement functions. Properly understood, however, this universality is very intuitive, and can be understood by analogy with another set of coordinates with a familiar, universal behavior: derivatives. Consider the set of coordinates $\{y_n\}$ defined by
\begin{align*}
    y_n(\textbf{x}) = \frac{d^n}{dt^n} g \circ \textbf{F}_t (\textbf{x}) \bigg|_{t = 0}
\end{align*}
The action of the Koopman operator on these `derivative coordinates' has a very simple form, which is universal by construction:
\begin{align*}
    \mathcal{K} y_{n} = y_{n+1}
\end{align*}
A similar representation can be found for shift-coordinates
\begin{align*}
    y_n(\textbf{x}) &= g \circ \textbf{F}_{n\Delta t}(\textbf{x}) = \mathcal{K}_{n\Delta t} g(\textbf{x})
\end{align*}
where the action of the Koopman operator $\mathcal{K}_{\Delta t}$ is given by
\begin{align*}
    \mathcal{K}_{\Delta t} y_n &= y_{n + 1}
\end{align*}

The underlying reason for these system-independent representations is that the dynamic information of the system is encoded in the trajectory $g(\textbf{x}(t))$, on which the same functionals are applied to construct the coordinate functions. Convolutional coordinates are likewise defined by taking functionals on segments of trajectories, a process which can be performed identically regardless of the underlying system that generated the trajectory data.

\section{SVD Convolutional Coordinates}
In practice, we will not have access to an infinite set of coordinates with which to embed our signal. We will generally be able to keep track of only a finite number of coordinates at any given time. Furthermore, we will usually not be working with ideal smooth trajectories, but instead with discretized trajectories limited by a finite sampling frequency. This latter condition puts a fundamental limit on the number of linearly independent coordinates that we can generate that are still {\em smooth} enough for finite difference approximations to effectively approximate the derivative of each coordinate basis vector. Since we can only consider a finite set of coordinates, it is imperative that we choose a set that effectively encodes the dynamics of our system. In general, an arbitrary basis will not be suitable. The reason for this is that the fixed linear relations are usually too rigid to effectively encode the dynamics without higher-order corrections. The spectrum of the estimated finite-dimensional linear system on the convolution coordinates depends strictly on the choice of basis, and will generally not match that of the underlying system. This is a fundamental issue, particularly when the derived models have unstable eigenvalues. To circumvent this issue, we need to pick a problem-specific basis.

Brunton et al. considered Hankel-SVD coordinates for representing the Koopman operator in \cite{bruntonproctorkaiserkutz2017}. They found that the action of the Koopman operators for a wide variety of systems, included ones with continuous spectra, could be well-approximated by applying DMD to these coordinates. However, their results were primarily empirical, without providing any theoretical guarantees on accuracy. Arbabi and Mezi\'c ~\cite{arbabimezic2017} later studied some properties of these coordinates, but did not apply them to DMD directly. Motivated by the successes of these works, we propose analyzing the properties of the continuous generalization of the Hankel-SVD basis for convolutional coordinates.

The notion of taking the SVD of a continuous function goes back almost a century; for a useful review of the relevant ideas, see \cite{townsendtrefethen2015}. We make use of the following theorem from this review as a basis for what follows:
\begin{theorem}\label{townsend_thm_1}
 For an arbitrary Lipschitz-continuous function $f(s,t)$ on $[a,b] \times [c,d]$, there exists:
\begin{itemize}
    \item an orthonormal basis $\{u_k\}$ on $[a,b]$,
    \item an orthonormal basis $\{v_k\}$ on $[c,d]$,
    \item a unique, nonnegative, monotonically decreasing sequence $\sigma_k$ such that $\sigma_k \to 0$,
\end{itemize}
such that the following series converges absolutely and uniformly for all $s,t$:
\begin{equation}
    f(s,t) = \sum_{i = 0}^\infty \sigma_i u_i(s) v^*_i(t)
\end{equation}
Furthermore, the choice of basis functions $u_k$ and $v_k$ is unique for all $\sigma_k > 0$, up to an overall complex sign. 
\end{theorem}
We can apply this decomposition to a trajectory drawn from a dynamical system of interest. Consider an arbitrary point $\textbf{x} \in \mathbb{R}^n$ and let $\Omega$ be the set of points $\textbf{F}_t(\textbf{x})$ for all $t \in [0,T]$. The above theorem implies the existence of orthonormal bases $\textbf{u}_i : [-\tau,\tau] \to \mathbb{R}^k$ and $v_i : [0,T] \to \mathbb{R}$ such that 
\begin{equation}\label{eq:continuousSVD}
    \textbf{g} \circ \textbf{F}_{t + s}(\textbf{x}) = \sum_{i = 0}^\infty \sigma_i \textbf{u}_i(s, \textbf{x}) v^*_i(t, \textbf{x})
\end{equation}
for positive $\sigma_i$. We will refer to these $\textbf{u}_k$ functions as `SVD basis functions.' This decomposition is the continuous analog of the SVD of the Hankel matrix in (\ref{eq:hankelsvd}). It follows from the orthogonality of the functions $v_j(t, \textbf{x})$ that the convolutional coordinates for the basis $\textbf{u}_j$ on the domain $\Omega$ are given by
\begin{equation}
    w_j(\textbf{F}_t(\textbf{x})) = \sigma_j v_j(t, \textbf{x})
\end{equation}

In the event that the system is ergodic and has an invariant measure $\mu$, the orthonormal basis $\textbf{u}_j$ will converge to POD on delays with respect to the invariant measure as $T \to \infty$ \cite{arbabimezic2017}. In this regime, the basis functions $\textbf{u}_j$ are independent of $\textbf{x}$ and can be given as the eigenfunctions of the autocorrelation operator, i.e. the functions satisfying the following integral equation:
\begin{equation}
    \int_{-\tau}^\tau \textbf{C}(s, s')\textbf{u}_j(s')\,ds' = \lambda_j \textbf{u}_j(s)
\end{equation}
where
\begin{equation}
    \textbf{C}(s,s') = \int \mathcal{K}_s \textbf{g}(\textbf{x}) \otimes \mathcal{K}_{s'} \textbf{g}(\textbf{x})\,d\mu
\end{equation}

We show that these coordinates have many interesting properties, including the following:
\begin{enumerate}
    \item The predicted infinite-dimensional Koopman representation $K_{jk} = \langle {\bf u}_j, {\bf u}_k' \rangle$ for this basis match the least-squares estimate for these coefficients across time when estimated on $[0,T]$, and thus matches the output of the EDMD algorithm applied to the convolutional coordinates.
    \item When $\textbf{g}$ is an observable consisting of a linear combination of $r$ eigenfunctions, these coordinates are `optimal' in the sense that exactly the first $r$ coordinates will span a Koopman-invariant subspace.
\end{enumerate}
These properties motivate the application of these coordinates to dynamical systems, which is demonstrated in Sec.~\ref{Sec:Applications}.

\subsection{Spectral Dynamics in Delay Coordinates: Continuous and Discrete  Formulation}
The basis functions $\{\mathbf{u}_j(s)\}_{j=0}^{\infty}$ are an {\em optimal} basis for representing the delay embedding $\mathbf{x}(t + s)$ in the sense that the first $r$ coordinates provide the closest subspace to the full space of coordinates. This basis also has a number of other attractive properties from a dynamical systems perspective. In particular, truncating the infinite-dimensional linear dynamics of the system (\ref{eq:convdifferentialdynamics}) provides the best least-squares approximation to the true dynamics of the system estimated from finite trajectory data. We prove this result below:

\begin{theorem}\label{thm:DMDconvconnection} Consider a point $\textbf{x}$ and let $\mathbf{v}(t,\textbf{x}) = \mathbf{v}(t):=[v_1,v_2,\ldots,v_r]^T$ be the vector consisting of the first $r$ coordinates $v_j(t,\textbf{x})$. Then the coefficients of the linear map $\bT$ advancing the $\mathbf{v}(t)$, i.e.\ $\mathbf{v}'(t) = \bT\mathbf{v}(t)$, minimizing the squared error
\begin{equation}
    E_{S}(\bT) = \langle\bT\mathbf{v} - \mathbf{v}',\bT\mathbf{v} - \mathbf{v}'\rangle_t = \int_0^T \Vert\bT\mathbf{v}(t) - \mathbf{v}'(t)\Vert_2^2\,dt,
\end{equation}
are given by
\begin{equation}\label{eq:AVcomponents}
    T_{jk} = \frac{\sigma_k}{\sigma_j}K_{jk}.
\end{equation}
\end{theorem}
\begin{proof}
Varying $E_{S}(\bT)$ with respect to $T_{jk}$ gives
\begin{align*}
    \frac{\partial}{\partial T_{jk}} E_{S}(\bT) &= 2 \int_{0}^T (\bT\textbf{v}(t) - \textbf{v}'(t)) \cdot \delta_{jk} \textbf{v}(t)\,dt\\
    &= 2\int_0^T v_k(t) \bigg(\sum_{i = 1}^N T_{ji} v_i(t) - v_j'(t)\bigg)\,dt\\
    &= 2 T_{jk} - 2\int_{0}^T v_k(t) v_j'(t)\,dt
\end{align*}
Setting this to zero, we obtain
\begin{align*}
    T_{jk} &= \langle v_k, v_j' \rangle_{t}
\end{align*}
Noting that $v_j = \frac{w_j}{\sigma_j}$, and applying theorem \ref{maintheorem}, we obtain
\begin{align*}
    T_{jk} &= \langle v_k, \sum_{i = 1}^\infty \frac{\sigma_i}{\sigma_j} K_{ji} v_i \rangle_t\\
    &= \frac{\sigma_k}{\sigma_j} K_{jk}
\end{align*}
\end{proof}

\begin{remark}
This result holds identically for the linear map minimizing $\Vert \bT\mathbf{v}(t) - \mathbf{v}(t + \Delta t)\Vert_2$, with the discrete time evolution coefficients given in \eqref{eq:convdiscretedynamics}, provided the assumptions of theorem \ref{thm:analytictheorem} hold. This linear map coincides with the map estimated by DMD on the convolutional coordinates.
\end{remark}
\begin{remark}
This result also shows that the closest linear approximation of the associated convolutional coordinates $w_j(\textbf{F}_t(\textbf{x})) = \sigma_j v_j(t,\textbf{x})$ or $\bw = \bSigma\bv$, are given simply by $\bSigma^{-1} \bT \bSigma = \textbf{K}$, the exact rank-$r$ truncation of the analytical form of the Koopman operator in (\ref{eq:convdifferentialdynamics}).
\end{remark}

Another useful property of the SVD convolutional coordinates is that the Koopman eigenfunctions and eigenvalues of systems that admit a finite Koopman mode decomposition can be exactly recovered from the projections of the dynamics onto a finite set of singular vectors. Suppose that $\mathcal{K}_t \textbf{g}(\textbf{x}_0)$ admits the eigenvalue decomposition
\begin{equation}
    \mathcal{K}_t \textbf{g}(\textbf{x}_0) = \sum_{j=0}^{r} \mathbf{a}_{j} e^{\lambda_j t}
\end{equation}
where the set of eigenvalues $\{\lambda_j\}_{j=0}^r$ 
is finite, i.e. has cardinality $|\{\lambda_j\}_{j=0}^r| = r$.
For most observables on general nonlinear systems, even those in possession of a pure point spectrum, there will be no finite $r$ for which this relationship holds, except in the special case where $g$ is a finite linear combination of Koopman eigenfunctions. However, even if the sum is infinite, the coefficients $\mathbf{a}_{j}$ will generally vanish as $j \to \infty$, so we can obtain arbitrarily good approximations of these systems with large but finite truncations of the Koopman mode decomposition. To obtain the time-delay embedding of the trajectory, we compute $\mathcal{K}_t \textbf{g}(\textbf{x}_0)$:
\begin{equation}
    \mathcal{K}_{t+s} \textbf{g}(\textbf{x}_0) = \sum_{j=0}^{r} \mathbf{a}_{j} e^{\lambda_j t} e^{\lambda_j s}.
\end{equation}
Thus, in the delay-embedded space, the dynamical evolution is given by $\sum_{j=0}^{r} \bb_j e^{\lambda_j t}$, where $\bb_j = \mathbf{a}_{j} e^{\lambda_j s}$. Thus the delay-embedded state lies in the span of $\{e^{\lambda_j s}\}_{j=0}^r$. This is a finite-dimensional subspace and thus these windows are also in span of the first $|\{\lambda\}|$ singular vectors. Since the singular vectors and the exponential vectors are related by a finite change of basis, it follows the dynamics (and in particular the spectrum) of the system are encoded exactly in the map $\bT$ of these convolutional coordinates. Furthermore, if the eigendecomposition of $\bT$ is given by $\bT = \bP\pmb{\Lambda}\bP^{-1}$, then the functions
\begin{equation}
    b_j(\mathbf{x}) = (\bP^{-1}\bw(\mathbf{x}))_j
\end{equation}
are eigenfunctions of the Koopman operator with associated eigenvalue $\lambda_j$. These can also be rewritten as follows:
\begin{equation}
    b_j(\mathbf{x}) = \int_{-\tau}^\tau \big\langle \sum_{k=0}^{\infty} (\bP^{-1})_{jk}\mathbf{u}_k(s), \mathcal{K}_{s}g(\textbf{x})\big\rangle\,ds.
\end{equation}
The eigenfunctions of the Koopman operator can thus be interpreted as convolutional coordinates of an associated {\em eigenfilter} basis $\sum_{k=0}^{r} (\bP^{-1})_{jk} \mathbf{u}_k(s)$ with the system trajectory. Note that in general these filters are not orthogonal, and thus the results of theorem \ref{maintheorem} are not applicable in this case.

While this result is useful, we may also be interested in understanding the SVD coordinate approximations to the Koopman operator without restrictions to cases where the system has purely discrete spectra. The following theorem helps guide our understanding of the structure of the resultant approximations: 

\begin{theorem}\label{thm:SVDantisymmetric} Consider any $\textbf{x}$, and let $\mathbf{g}\circ \mathbf{F}_t(\textbf{x})$ be bounded for all $t \in [-\tau,T+\tau]$. Let $\textbf{u}_j, \sigma_j$ and $v_j$ be the singular value decomposition of $\mathbf{g} \circ \mathbf{F}_{t+s}(\textbf{x})$ on $[-\tau, \tau] \times [0,T]$. Suppose in the limit as $T \to \infty$ that $\sigma_j \to \infty$ holds for the first $n$ singular values, and suppose that the limit $\sigma_{j}/\sigma_k$ holds for all singular values $\sigma_j, \sigma_k$ in the first $n$ singular values. Then in the limit as $T\to \infty$ the map given in (\ref{eq:AVcomponents}) on the first $n$ SVD convolutional coordinates is antisymmetric.
\end{theorem}
\begin{proof}
Consider the product $v_j(t)v_k(t)$ for $j, k \leq n$. Differentiating this gives
\begin{align*}
    \frac{d}{dt} v_j v_k &= v_j' v_k + v_k' v_j.
\end{align*}
Integrating both sides from 0 to $T$ and substituting the coefficients (\ref{eq:AVcomponents}), we are left with
\begin{equation}
    v_j(T)v_k(T) - v_j(0) v_k(0) = T_{jk} + T_{kj}.
\end{equation}
The coordinate $v_j(t)$ is normalized over the interval $[0,T]$, and the normalization factor is given by $\frac{1}{\sigma_j}$. Thus the following bound holds:
\begin{align*}
    |v_j(t)| &= \bigg|\frac{\int_{-\tau}^\tau \langle \mathbf{u}_j(s), \mathcal{K}_{s + t} \textbf{g}(\textbf{x})\rangle\,ds}{\sigma_j}\bigg|
    \leq \frac{1}{\sigma_j}\Vert \mathbf{u}_j \Vert \Vert \mathcal{K}_{s + t} \textbf{g}(\textbf{x}) \Vert.
\end{align*}
Since $\Vert \mathbf{u}_j \Vert$ and $\Vert \mathcal{K}_{t + s} \textbf{g}(\textbf{x}) \Vert$ are uniformly bounded for all time lengths $T$, it follows that $|v_j| \to 0$ as $T \to \infty$. Therefore in the limit as $T \to \infty$ we obtain
\begin{equation}
    T_{jk} = -T_{kj}.
\end{equation}
\end{proof}

One disadvantage of using principal components is that their average magnitude will depend on the estimation length $T$ due to the normalization condition. This is problematic because as $T\to\infty$ these trajectories tend to zero magnitude. Instead, we will usually want to model the respective convolutional coordinates, which are the non-normalized principal components. The linear model estimated for this system is related to the linear model for the principal components by the change of basis relation $\hat{\bT} = \bSigma^{-1} \bT \bSigma$, or in component form $\hat{T}_{jk} = K_{jk} = T_{jk} \frac{\sigma_j}{\sigma_k}$. This transformed matrix will generally not be antisymmetric; however, the spectrum of $\hat{\bT}$ will be identical to the spectrum of $\bT$, which converges to the imaginary axis in the limit as $T \to \infty$. This result indicates that SVD coordinates may be effective at approximating systems with Koopman operators with purely imaginary spectra.

These properties give us the following corollary, which may have useful implications for understanding the spectral quality of these approximations to the Koopman operator:
\begin{corollary}
Let $i\alpha_1, i\alpha_2, \cdots, i\alpha_r$ with $\alpha_i \leq \alpha_j$ if $i < j$ be the DMD eigenvalues for the first $r$ convolutional coordinates, and let $i\beta_1, i\beta_2, \cdots, i\beta_{r+1}$ with $\beta_i \leq \beta_j$ if $i \leq j$ be the DMD eigenvalues for the first $r + 1$ convolutional coordinates. It follows from Cauchy's interleaving theorem that
\begin{equation}
    \beta_i \leq \alpha_i \leq \beta_{i+1}.
\end{equation}
\end{corollary}





\subsection{A Conjecture on the Error of HAVOK Approximations to the Koopman Operator}

From the result of theorem \ref{thm:DMDconvconnection} we can write $E_{RMS}$ for the for the convolutional coordinates as
\begin{align*}
    E_{RMS} &= \frac{1}{T}\sqrt{\int_0^T \sum_{j = 1}^N \bigg(\sum_{k > N} K_{jk} w_j(t) \bigg)^2 dt}\\
    &= \frac{1}{T}\sqrt{\sum_{j = 1}^N \sum_{k > N} \sum_{l > N} \sigma_k \sigma_l K_{jk}K_{jl} \int_0^T v_j(t)v_l(t)\,dt}\\
    &= \frac{1}{T} \sqrt{\sum_{j = 1}^N \sum_{k > N} \sigma_k^2 K_{jk}^2}.
\end{align*}
The rate of convergence of this error is related to the respective rates of increase and decay of $K_{jk}^2$ and $\sigma_k^2$. It has been empirically observed that the singular values of many systems of interest decay exponentially with $k$. It has also been formally shown that the decay of the singular values is related to the smoothness of the function $\textbf{x}(t + p) = \textbf{F}_{t+p}(\textbf{x})$. We paraphrase theorem 7.1 from \cite{townsendtrefethen2015} which gives bounds on $\sigma_j$ in terms of the derivatives of $\textbf{x}(t + p)$:
\begin{theorem}
If, for some $v \geq 1$, the functions $x_t(p)$ have a $v$th derivative of variation uniformly bounded with respect to
$v$, or if the corresponding assumption holds with the roles of $t$ and $p$ interchanged, then the singular values
and approximation errors satisfy $\sigma_k = O(k^{-v})$. If, for some $\rho > 1$, the functions $\textbf{x}(t + p)$ can be extended
in the complex $t$-plane to analytic functions in the Bernstein $\rho$-ellipse scaled to $[-\tau,\tau]$ uniformly bounded
with respect to $p$, then the
singular values and approximation errors satisfy $\sigma_k = O(\rho^{-k})$.
\end{theorem}
While this addresses the question of the decrease of the singular values, the question of the increase in $K_{jk}^2 = \langle {\bf u}_j, {\bf u}_k'\rangle^2$ is not well studied. To our knowledge, no results currently exist on the smoothness of the functions ${\bf u}_k$, nor on how their derivatives ${\bf u}_k'$ scale with $k$. However, our empirical results are promising. In section 4, we will show empirically that the coefficient $K_{jk}$ appear to scale only polynomially with $k$. We conjecture that a polynomial bound exists on the derivatives of the singular components in terms of $k$ and thus that a polynomial or exponential bound exists on the total error $E_{RMS}$ in terms of $k$.

If proven, this result could have significant implications for Koopman theory. It has been shown that the projections of the Koopman operator onto generic $D$-dimensional subspaces of observables converges to the Koopman operator as $D \to \infty$ \cite{kordamezic2017JNS}. However, for finite $D$, the quality of these approximations is highly dependent on the choice of subspace, and for a generic family of observables it may be quite difficult to obtain convergence without exceedingly high dimensional states. A bound of the form conjectured above would guarantee that these representations converge quickly to a known precision in convolutional SVD coordinates for any (sufficiently smooth) dynamical system. We do not know of any other family of observables for which a similar bound exists.

\subsection{Universal Representations: SVD Convolutional Coordinates}

In general, the basis functions $\textbf{u}_j$ depend sensitively on the underlying system. In principle, one can construct convolutional coordinates using these basis functions for dynamical systems other than the one they were derived from, and the action Koopman operator for that system on those coordinates can still be represented by \ref{eq:convdifferentialdynamics}. However, these coordinates will no longer have the appealing properties of SVD convolutional coordinates.

There are, however, two interesting places where these coordinates exhibit universal behaviors: in the long-$\tau$ and short-$\tau$ limits. In the long-$\tau$ limit, it has been shown that, for a stationary time series, the SVD basis functions will converge to a Fourier basis \cite{ssavolcanic}. Conversely, in the short-$\tau$ limit, it has been shown that the SVD basis functions converge to a basis of Legendre polynomials \cite{Gibson1992phD}, provided the signal and its derivatives are bounded.

Unfortunately, while the SVD basis functions converge to these two bases in the short- and long-$\tau$ limits, the off-diagonal correlations $\langle w_j, w_k\rangle/\Vert{w_j}\Vert \Vert w_k \Vert$ generally remain $O(1)$ for Fourier and Legendre bases, even as the window is brought to its limits. Further corrections are needed to fully diagonalise the correlation matrix at any window length. In section \ref{subsec:lorenz}, we investigate the connections between the Legendre basis and the empirical SVD bases for the Lorenz system in the small-$\tau$ regime, and find that, while the corrections to remove off-diagonal correlations are significant, certain qualitative features of the model are related to the Legendre basis model. Namely, the Koopman representation in SVD convolutional coordinates is roughly an `antisymmetrization' of the Koopman representation in Legendre coordinates. In appendix \ref{ap:Legendre Basis}, we derive analytic expressions for the representation of the Koopman operator in the associated SVD convolutional coordinates for these two bases. Furthermore, we investigate the connections between the Legendre basis and the empirical SVD bases for the Lorenz system in section \ref{subsec:lorenz}.

\subsection{Computing SVD Convolutional Coordinates from Data}\label{sec:SVDcoordAlgo}
In applications, we generally do not have access to the full trajectory $\mathbf{x}(t)$, but instead discretely sampled signal $\bx_k:=\bx(t_k)$ sampled with timestep $\Delta t = t_{k+1}-t_k$. We employ the following numerical approximations to compute the quantities relevant to our theory and as demonstrated for the examples in Sec.~\ref{Sec:Applications}:
\begin{enumerate}
	\item We first construct the Hankel matrix~\eqref{eq:hankelsvd} from a sampled trajectory of the considered dynamical system. The dimension of this matrix is $DN \times M$, where $D$ is the number of observables, $N$ is the delay embedding dimension and $M$ is the number of snapshots. Then the SVD of this matrix is computed to obtain approximations of $\mathbf{u}_j(s)$ of (\ref{eq:continuousSVD}) and $v_j(t)$. In general, only the first few singular vectors $\mathbf{u}_j(s)$ will be well-approximated by the singular vectors $\mathbf{u}_j$, due to the exponential decay of the singular values. Instead of computing a full or thin SVD of $\bH$, we can instead compute a partial SVD of some rank $r \ll \min(N,M)$. This approach is highly efficient without sacrificing any information about the dynamics. We also suggest an alternative approach for computing these observables and their dynamics by approximating the autocorrelation function, which is discussed in \ref{ap:autocorrelation_approximation}.

	\item We can compute the matrix elements $K_{jk}$ either by computing $\frac{\sigma_j}{\sigma_k}\langle \bv_j, {\bv}_k'\rangle$ or $\langle \mathbf{u}_j, {\bu}_k'\rangle$. These approaches are of order $O(r^2M)$ and $O(r^2DN)$, respectively, and can be used when the measurement states are are high- and low-dimensional, respectively. We compute the derivatives ${\bu}_j'$ and ${\bv}_j'$ using numerical finite differencing.
	\item The dynamics of the convolutional coordinates are then approximated by
	\begin{equation}
	w_j(t) = \sum_{j = -\tau/\delta}^{\tau/\delta} \mathbf{u}_{j}\mathbf{y}_{t + k}.
	\end{equation}
\end{enumerate}

\section{Applications to Dynamical Systems}\label{Sec:Applications}

\subsection{Linear Systems}\label{sec:linsys}
\begin{figure}\label{fig:linresults}
    \centering
    \begin{overpic}[width=0.49\textwidth]{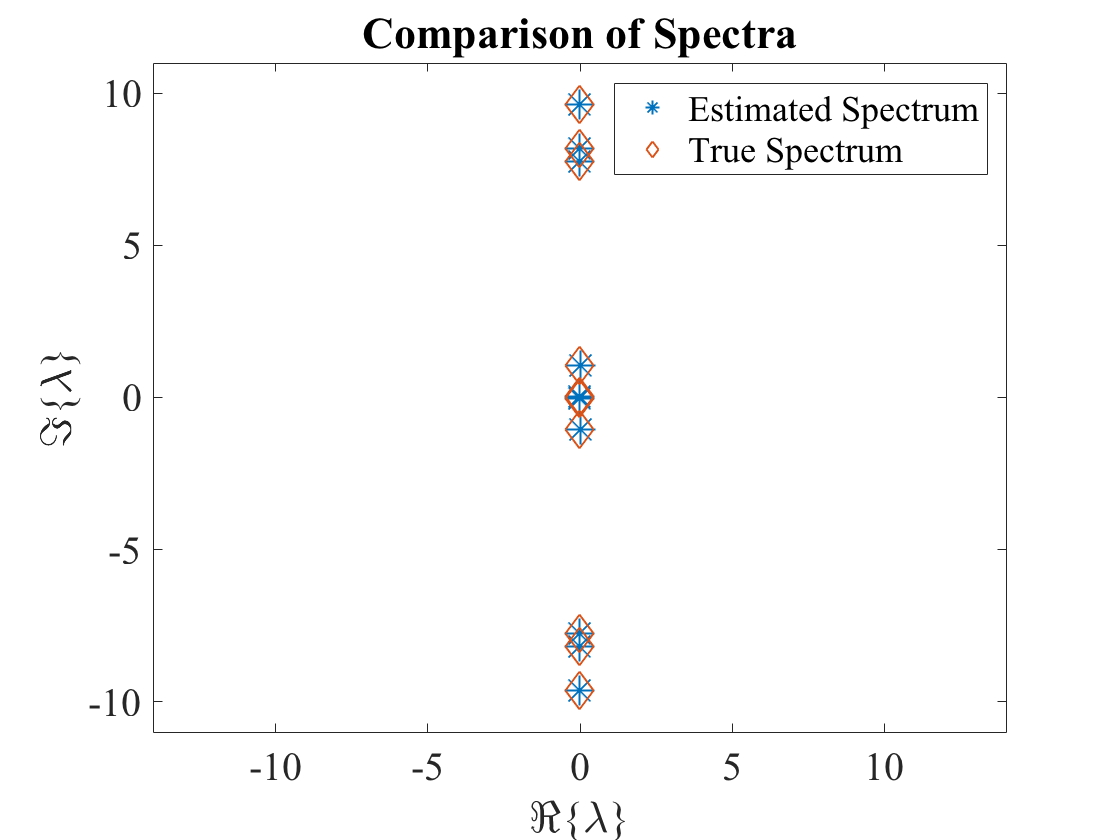}
    	\put(0,0){(a)}
    \end{overpic}
    \begin{overpic}[width=0.49\textwidth]{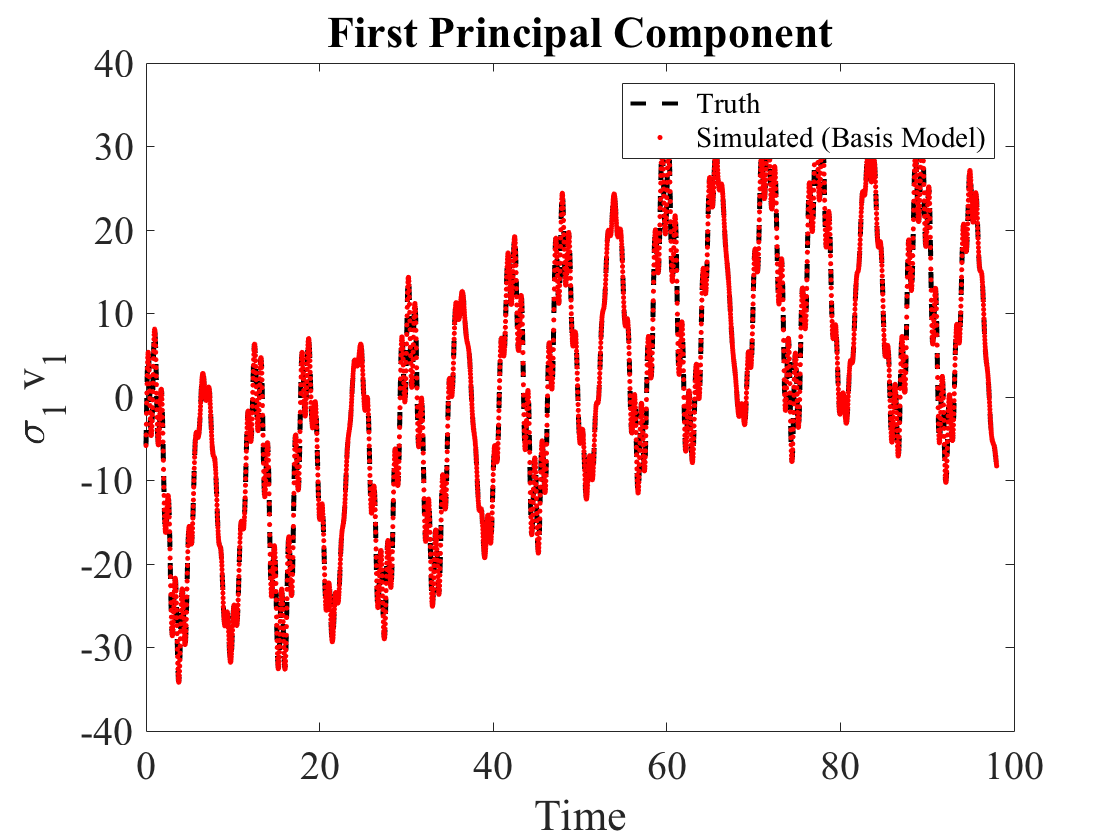}
    	\put(0,0){(b)}	
    \end{overpic}
    \caption{Delay embedding of a trajectory from a linear system with random imaginary eigenvalues: 
    	(a) reconstructed spectrum matches the true spectrum almost exactly, and
    	(b) the true trajectory and reconstructed trajectory of the first convolutional coordinate fir with negligible error. 
    }
    \label{fig:my_label}
\end{figure}
The simplest possible setting for a Koopman analysis using a SVD delay embedding is a finite-dimensional linear system. As established previously, the dynamics of these systems are spectrally identical to the dynamics in a finite number of convolutional coordinates. Thus, linear systems are useful for illustrating the main principle of this method, as well as for highlighting some of the issues that arise when working with discrete signals instead of continuous functions. In particular, we consider a number of numerical experiments with real trajectories taken from linear systems with random imaginary spectra. A single coordinate is measured and then delay embedded. The SVD observables and associated linear models are computed using the algorithm in Sec.~\ref{sec:SVDcoordAlgo}. We find that our method is able to exactly reconstruct these trajectories in almost all cases. A typical result is illustrated in Fig.~\ref{sec:linsys}.

A number of pathological cases exist for which our method does not perform as well as expected. The most fundamental case is that our method performs poorly if the frequencies in the dynamics are {\em close} with respect to $\tau$, i.e. when $(\omega_j - \omega_k)\tau$ is small for a subset of frequencies so that the eigenvalues are nearly degenerate. 
While these eigenvalues would be closely matched in the linear model, the simulated trajectory using this model would often fail to match that of the original system. This can be explained as an effect of poor conditioning. Recall that the normalized eigenvectors associated with a frequency $\omega$ are
\begin{align*}
    v = \frac{1}{\sqrt{2\tau}}e^{i\omega t}.
\end{align*}
The inner product between $v_j$ and $v_k$ is then given by
\begin{align*}
    \langle v_k, v_j \rangle &= \frac{1}{2\tau} \int_{-\tau}^\tau e^{-i\omega_k t}e^{i\omega_j t}\,dt = \frac{1}{2\tau}\frac{e^{i(\omega_j - \omega_k)t}}{i(\omega_j - \omega_k)} \bigg|_{-\tau}^\tau = \sinc((\omega_j - \omega_k)\tau).
\end{align*}
As $(\omega_j - \omega_k)\tau \to 0$ this term goes to 1, indicating that the eigenspaces converge. In this regime, the linear system is ill-conditioned, and small errors in the estimation of the eigenvectors can propagate catastrophically. Then, it may be necessary to avoid estimating the model directly from the SVD basis in order to obtain a more precise estimate. More generally, if the $r$ eigenvectors occupy a smaller and smaller volume of state space, then the variance of their distribution becomes smaller and the singular values decay rapidly. This spectral crowding, i.e. many closely spaced eigenvalues, makes it difficult to resolve the dynamics of the system when the eigenvectors are close together. In such a regime, for improved robustness, we recommend computing DMD model from the full trajectories of the convolutional coordinates $w_j$, rather than the basis vectors $\mathbf{u}_j$.

\subsection{The Van der Pol Oscillator}
\begin{figure}
    \centering
    \includegraphics[width=\textwidth]{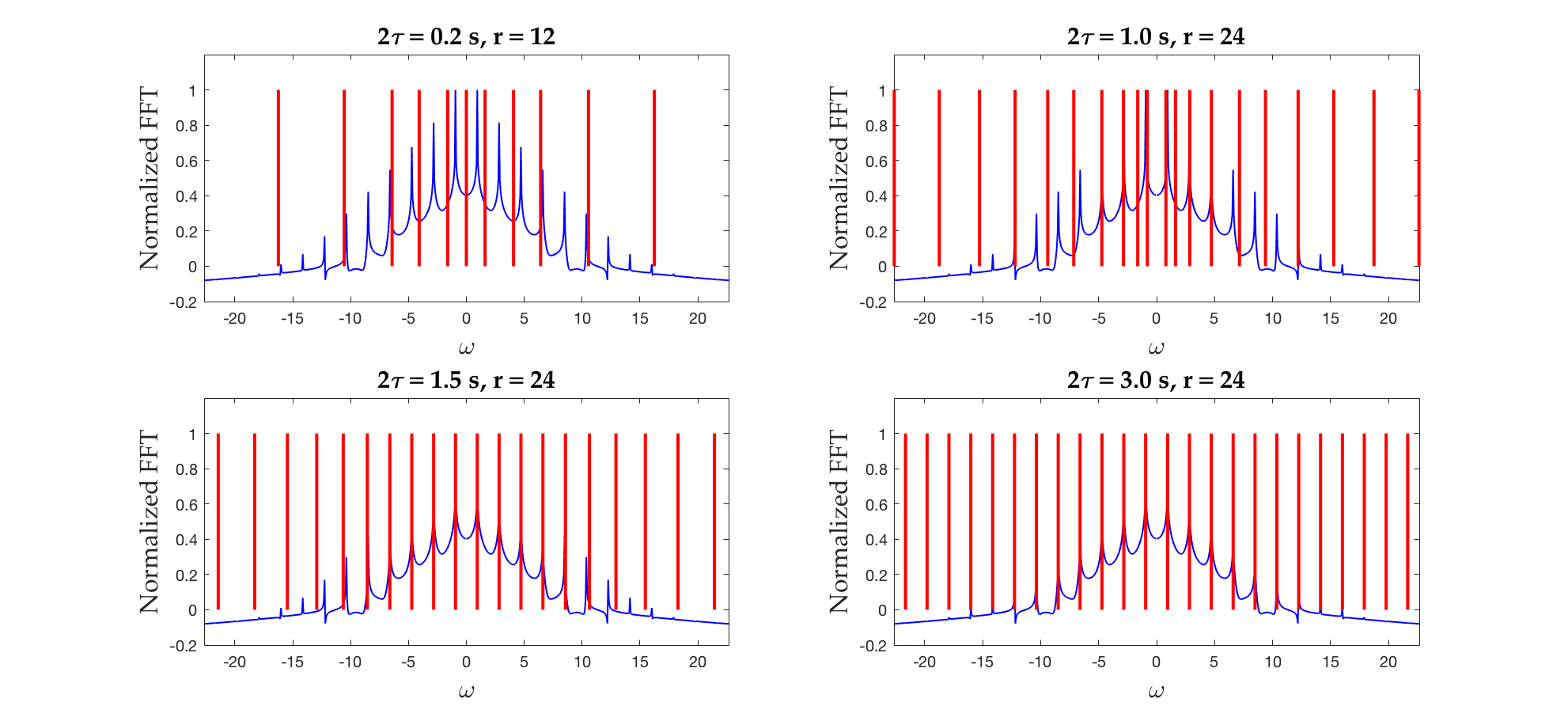}
    \caption{Comparison of the spectrum of the SVD coordinate operators on the Van der Pol attractor for $\mu = 1$ and $2\tau = 0.2, 1.0, 1.5$ and $3.0$. The Fourier transform of the trajectory is plotted in blue, and the locations of the imaginary component of the eigenvalues of the SVD coordinate operators are marked as red vertical lines. The spectrum of the operator matches the peaks in the Fourier spectrum for sufficiently large $\tau$. For smaller $\tau$, the spectrum appears distorted. This is likely related to the poor scaling properties of the eigenvectors discussed in section \ref{sec:linsys}.}
    \label{fig:vanderPolspectralcomparison}
\end{figure}
The van der Pol system is a nonlinear second-order differential equation:
\begin{equation}
    \frac{d^2x}{dt^2} - \mu(1 - x^2)\frac{dx}{dt} + x = 0.
\end{equation}
In the small-$\mu$ limit, the van der Pol system reduces to a harmonic oscillator. For positive $\mu$, a trajectory starting off the attractor decays asymptotically onto a limit cycle in the phase space spanned by $x$ and $\dot{x}:=\frac{d}{dt}x$. Since the limit cycle is periodic, we expect the spectrum of the Koopman operator on the attractor to be discrete integer multiples of the fundamental frequency $\omega$. Since the spectrum is also discrete, we expect that the SVD delay embedding method should be able to exactly reconstruct these frequencies. 

Importantly, nonlinearity in dynamical systems manifests two critical phenomenon:  (i) the production of harmonic frequencies, and (ii) shifts in the underlying frequencies as a function of the strength of the nonlinearity.  As will be shown by our time-delay embedding, the SVD coordinate system accurately extracts these manifestations.  For the van der Pol system, a classical asymptotic expansion in the weakly nonlinear limit using a Poincar\'e-Lindstedt expansion~\cite{bender2013advanced,kevorkian2013perturbation} with a {\em stretched} time coordinate is given by 
\begin{subeqnarray}
   && \tau=\omega(\epsilon)t=(\omega_0 +\epsilon \omega_1 + \hdots ) t,\\
   && x = x_0 + \epsilon x_1 + \epsilon^2 x_2 + \hdots ,
\end{subeqnarray}
where $\epsilon=\mu\ll 1$.
The Fredholm-Alternative theorem allows us to determine the asymptotic corrections to the leading order sinusoidal oscillations so that
\begin{equation}
  x(t) \approx 2 \cos \left[ (1 + 7\epsilon^2/16) t \right] + \epsilon \left[ \frac{3}{4} \sin \left[ (1 + 7\epsilon^2/16) t \right] 
  - \frac{1}{4} \sin \left[ 3 (1 + 7\epsilon^2/16) t \right] \right] + O(\epsilon^2) \, .
\end{equation}
Such asymptotic expansions not only allows one to compute the frequency shifts imposed by the nonlinearity,
i.e. from $\omega=1$ to $ \omega=(1 + 7\epsilon^2/16)$, but it also reveals the production of harmonics (the Koopman spectrum), as illustrated by the $\sin[3\omega(\epsilon)t]$ term at $O(\epsilon)$ which is generated by the cubic nonlinearity.  At $O(\epsilon^2)$, the nonlinearity generates a $\sin[5\omega(\epsilon)t]$ contribution and further corrections to the frequency.  Although asymptotic expansions are insightful, they are only valid in the weakly nonlinear regime.  The Koopman embedding proposed here can accurately compute the frequency shifts and harmonics (spectrum) generated from the nonlinearity even in the strongly nonlinear regime, allowing for improved analytic insight into nonlinear dynamical systems.

The van der Pol oscillator is simulated on the attractor for 100 time units with a time step of $\Delta t = 0.001$. We examine parameters $\mu = 0.1, 0.5, 1.0$ and $3.0$, which represent different regimes ranging from weakly to strongly nonlinear dynamics. In all of these cases, the eigenvalues of the SVD coordinate operators matched the dominant spectral peaks of the system, provided the delay embedding window $2\tau$ was chosen to be large enough. However, for sufficiently small $\tau$, the estimated spectrum did not match the peaks in the Fourier spectrum. This is likely related to the pathological behaviors with spectral crowding identified in section \ref{sec:linsys}. In particular, the stronger the nonlinearity $\mu$, the larger the window size needed to adequately represent the spectrum. This was likely due to two related phenomena: The peaks of the higher-order frequencies increase as the nonlinearity increases, resulting in more {\em active} frequencies and thus more crowding. Furthermore, the fundamental frequency $\omega$ decreases as $\mu$ increases, exacerbating the spectral crowding issue. A topic of further study will be to relate the minimum window size $2\tau$ to the magnitude of the nonlinearity.

\begin{figure}
    \centering
    \begin{overpic}[width=0.65\textwidth]{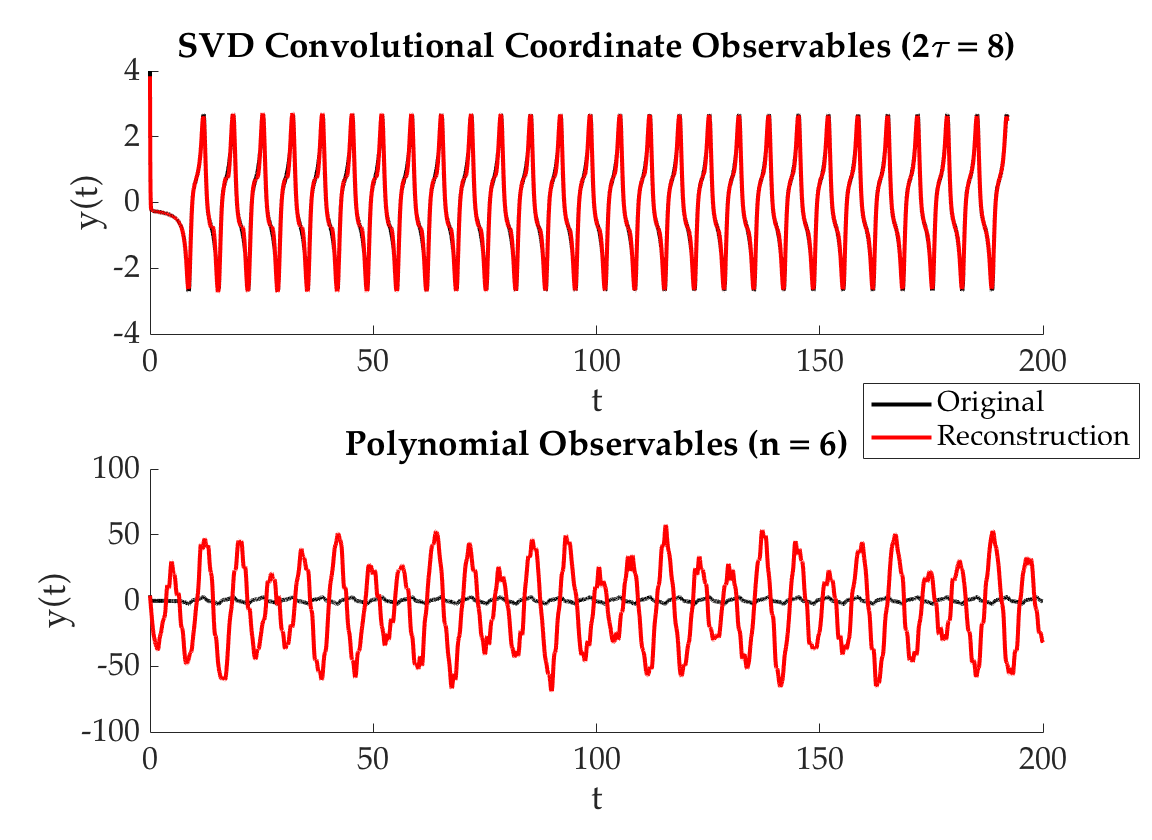}
    	\put(0,40){(a)}
    	\put(0,5){(b)}
    \end{overpic}	
    \caption{EDMD reconstruction of a trajectory from off the van der Pol attractor ($\mu = 1.0$) with (a) 28 SVD Convolutional coordinate observables, $2\tau = 8$, and (b) using polynomial observables up to order 6 (28 terms).}
    \label{fig:vanderpolreconstruct}
\end{figure}

Trajectories with initial conditions off of the attractor are analyzed. In this domain DMD-type linear models are not expected to perform well, since the spectrum of the Koopman operator is continuous in this regime, and thus no finite-dimensional linear model exists for these dynamics. The system is simulated with a randomly drawn initial condition off of the attractor and $\mu = 1.0$. For comparison, EDMD is applied to both a set of polynomial observables of order 6 and a set of SVD convolutional coordinates for varying window lengths. Instead of performing a delay-embedding with a single coordinate, we constructed the convolutional coordinates by delay embedding both the coordinates, $x$ and $\dot{x}$. The reconstruction achieved using polynomial observables appears uniformly poor. Surprisingly, we found that by taking a similar number of SVD convolutional coordinates, we could reconstruct the trajectory exactly, despite the continuous spectrum. This result is illustrated in Fig.~\ref{fig:vanderpolreconstruct}. The result is suggestive, but comes with the caveat that the window size necessary for such an accurate reconstruction is almost as long as the period of time it takes for the trajectory to decay onto the attractor. Nevertheless, this approach shows promise and may lead to future directions for approximating non-chaotic systems with continuous Koopman operator spectra.

\subsection{Nonlinear Schr\"odinger Equation}
In the previous examples, applications of SVD convolutional coordinates to dynamical systems of only low state dimension are considered. However, these methods can be applied to systems of arbitrary, possibly infinite dimension. As an example, we consider the nonlinear Schr\"odinger equation in one spatial direction:
\begin{equation}\label{eq:nls}
    iu_t + \frac{1}{2} u_{xx} + |u|^2 u = 0.
\end{equation}
This equation is nonlinear but admits soliton solutions that exhibit quasiperiodic behavior in time, and as such presents an interesting test case for Koopman spectral methods. Kutz et. al.~\cite{PDEKoopman_kutzetal2018complexity} studied these solutions using several DMD-based algorithms. They found that kernel-DMD methods generally performed quite poorly for a wide variety of kernel functions. This result was surprising, given the simple structure present in the soliton solutions. This negative result shows the importance of choosing a good subspace of observables for accurate reconstruction. They found further that by augmenting the state with the nonlinear observable $|u|^2u$, which was motivated from the nonlinearity appearing in the original equation, they were able to achieve reconstruction with high accuracy outperforming all other choices of observables.

\begin{figure}
    \centering
    \begin{overpic}[width=.9\textwidth]{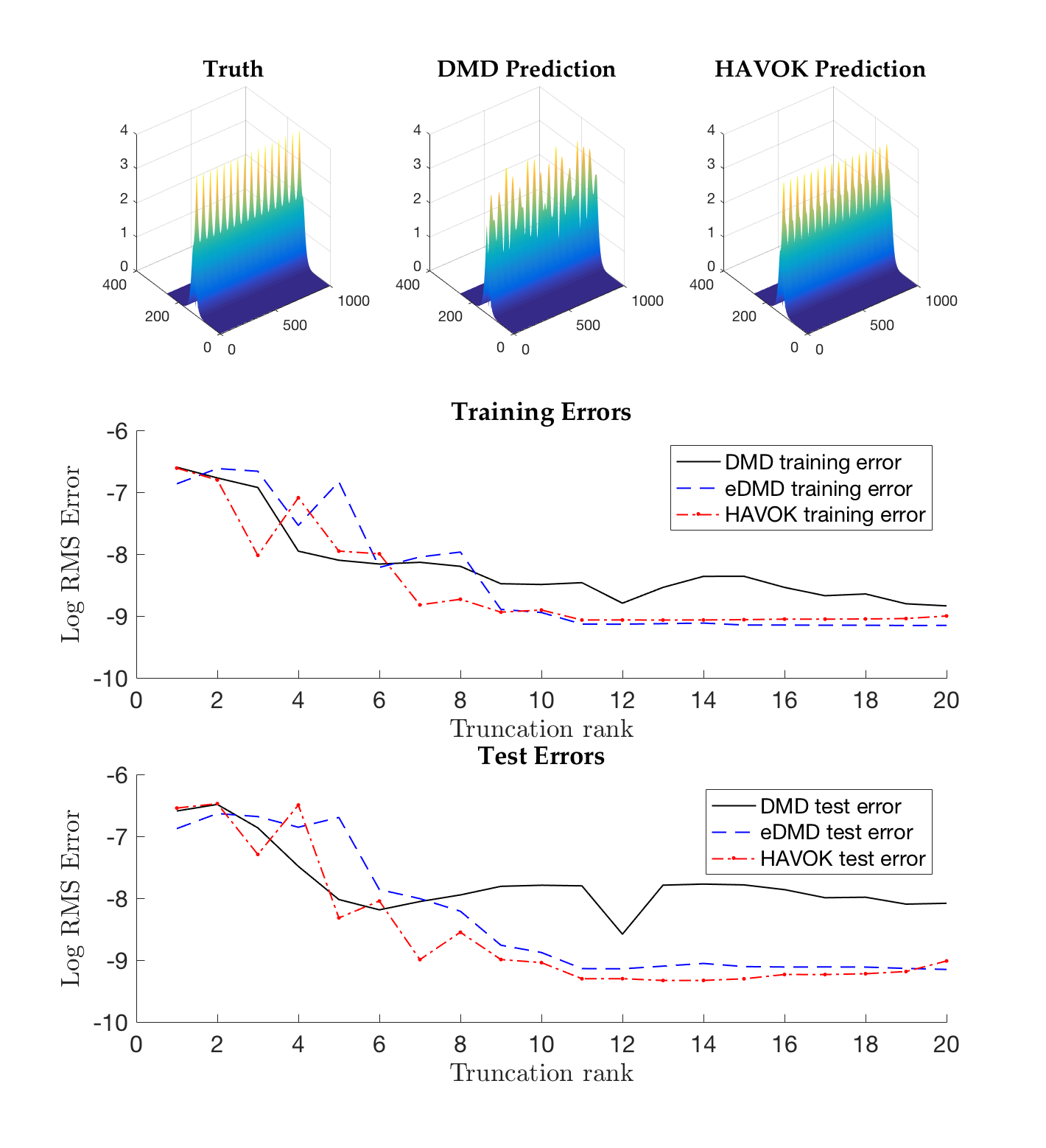}
    	\put(2,70){(a)}
    	\put(2,5){(b)}
    \end{overpic}	
    \vspace{-.4in}
    \caption{Prediction of the nonlinear Sch\"rodinger equation:
    	(a) True and predicted NLS solution using truncation rank $r = 14$. HAVOK achieves notably better reconstruction than DMD on linear observables. 
    	(b) Comparing the RMS prediction error of DMD, eDMD and HAVOK on the training and test data for increasing truncation rank. The SVD convolutional coordinate observables have similar training performance as eDMD with the physically motivated observable $u|u|^2$, and have the lowest test error of any of the methods (achieved at $r = 14$). Interestingly, the HAVOK test error increases after $r = 14$, while the eDMD test error remains steady.
  }
    \label{fig:nls_truncationrank_error}
\end{figure}

The nonlinear Schr\"odinger equation is simulated with the initial condition
\begin{equation}
    u(x,0) = 2\text{sech}(x),
\end{equation}
which is known to generate a soliton solution. The data is sampled across the spatial domain $x \in [-15,15]$, over a time domain of $t \in [0,16\pi]$, with 2000 time snapshots. 
The data is then splitted into a training set, on which the models are trained, and a test set, on which the models are validated by estimating the predictive accuracy of each method.
Due to the poor performance of kernel-DMD on these solutions, we chose instead to benchmark our approach using conventional DMD and EDMD with the system-motivated observable $|u|^2u$. 
We computed the reconstruction error of each model for different choices of the truncation rank $r$.

Summarizing the results, the HAVOK method accurately extracts the quasiperiodic dynamics of the system. In contrast, the DMD with linear observables is able to capture qualitatively the overall periodicity but is not quantitatively accurate and suffered from poorer performance on the test set than on the training set. The HAVOK model achieves the lowest RMS error of any of the methods tested, which was achieved at $r = 14$. Interestingly, the HAVOK test error increases after $r = 14$, while the eDMD test error remains steady. Overall, the performance of the HAVOK method is generally comparable with the performance of eDMD with the physically motivated observable $u|u|^2$. The SVD convolutional coordinates are not chosen using a priori physical knowledge, so the fact that they achieved comparable or better performance is encouraging.

\subsection{Understanding Intermittent Forcing in the Lorenz System}\label{subsec:lorenz}
\begin{figure}
    \centering
    \includegraphics[width=0.75\textwidth]{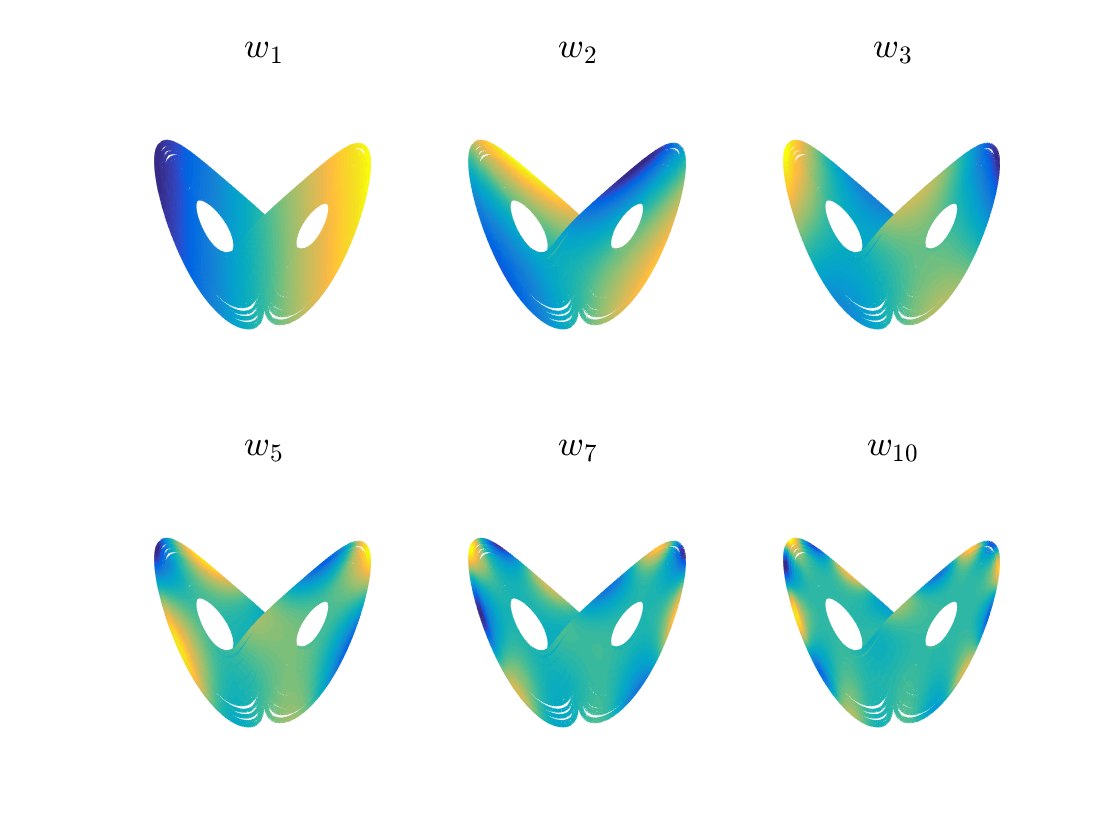}
    \vspace{-.3in}
    \caption{Convolutional coordinates $w_1$,$w_2$,$w_3$,$w_5$,$w_7$, and $w_{10}$ visualized on the Lorenz attractor.}
    \label{fig:lorenzconvcoords}
\end{figure}

It has been shown that chaotic systems do not admit exact finite-dimensional linear representations, as they possess mixed or purely continuous Koopman spectra \cite{Mezic2005nd}. In these cases, the Koopman operator may only have trivial (constant) eigenfunctions \cite{koopmanvonneumann1932}. Nevertheless, there is interest in obtaining an approximate linearization that, while not globally accurate, captures some essential features of the dynamics.

Brunton et. al. studied this question in \cite{bruntonproctorkaiserkutz2017} using the methodology of time-delay embedding and principal components. They computed the Hankel principal components for a number of chaotic systems and estimated a linear system by applying DMD to these coordinates. While the resulting linear models were not closed and did not approximately represent the dynamics by themselves, they found that they were able to approximately close these models to very high accuracy by taking the highest-order convolutional coordinate as a random exogenous input to the model and simulating the other coordinates based on the derived linear dependencies. The authors of this paper were unable to explain the apparent success of this method theoretically. However, the results in this paper provide a natural justification for these results, as well as a number of other striking features of their models.

We first reproduce their results in the case of the Lorenz system. The Lorenz system 
\begin{subequations}
	\begin{align}
	\dot{x}_1 &= 10(x_2-x_1)\\
	\dot{x}_2 &= x_1(28-x_3)-x_2\\
	\dot{x}_3 &= x_1x_2-8/3 x_3 
	\end{align}
\end{subequations}
is sampled for 100 time units with a time step $\Delta t = 0.001$. A Hankel matrix is constructed from the sampled trajectory with delay dimension 100. The SVD of the Hankel matrix yields the $\mathbf{U}, \mathbf{S}$ and $\mathbf{V}$ matrices. The trajectory in convolutional coordinates is then estimated as $\mathbf{S}\mathbf{V}^\dagger$. These results are plotted in figure \ref{fig:lorenzconvcoords}.

Instead of applying DMD to the principal components, the model is derived analytically from the $\mathbf{U}$ basis vectors. After normalization, this model matches the model derived in \cite{bruntonproctorkaiserkutz2017} with very high accuracy, both in the Frobenius norm and in the spectral norm. The structure of these models are illustrated in figure \ref{fig:matfig}. The normalized SVD coordinate operator is nearly antisymmetric, as predicted by theorem \ref{thm:SVDantisymmetric}. Interestingly, only the first off-diagonals are nearly nonzero, indicating that coefficients $\frac{\sigma_k}{\sigma_j} \langle u_j, u_k'\rangle$ drop off quickly with $j > k$. This fact is related to the relative growth rates of the singular values and the coefficients $\mathcal{A}_{jk}$. As conjectured previously, the decay rate of $\sigma_j$ is much stronger than the growth of $\mathcal{A}_{jk}$. This particular structure means that the derivative $\dot{v}_j$ can be well-approximated by $\mathbf{A}_{j(j-1)} v_{j-1} + \mathbf{A}_{j(j+1)} v_{jj+1}$. This dependency explains the {\em forcing} behavior observed in~\cite{bruntonproctorkaiserkutz2017}, in which the dynamics of the first $r$ coordinates could be modeled by a linear model on the first $r$ coordinates plus a linear forcing vector depending on the $r + 1$ coordinate.

It is instructive to compare the SVD coordinate model with the model predicted for Legendre convolutional coordinates, which are very close to the SVD coordinates \cite{Gibson1992phD}. The model for the Legendre coordinates is computed using the formulae derived in Appendix~\ref{ap:Legendre Basis}. Prior to normalization, the SVD coordinate and Legendre coordinate operators appear nearly identical, although the SVD coordinate one has a small negative lower-triangular component. After normalization, however, the matrices appear very different. The Legendre coordinate model matches the upper triangular component of the SVD model, but has no corresponding subdiagonal component. This results in a highly unstable model. The negative subdiagonal component of the SVD model stabilizes the model and gives it an imaginary spectrum.


\begin{figure}
    \centering
    \includegraphics[width=0.8\textwidth]{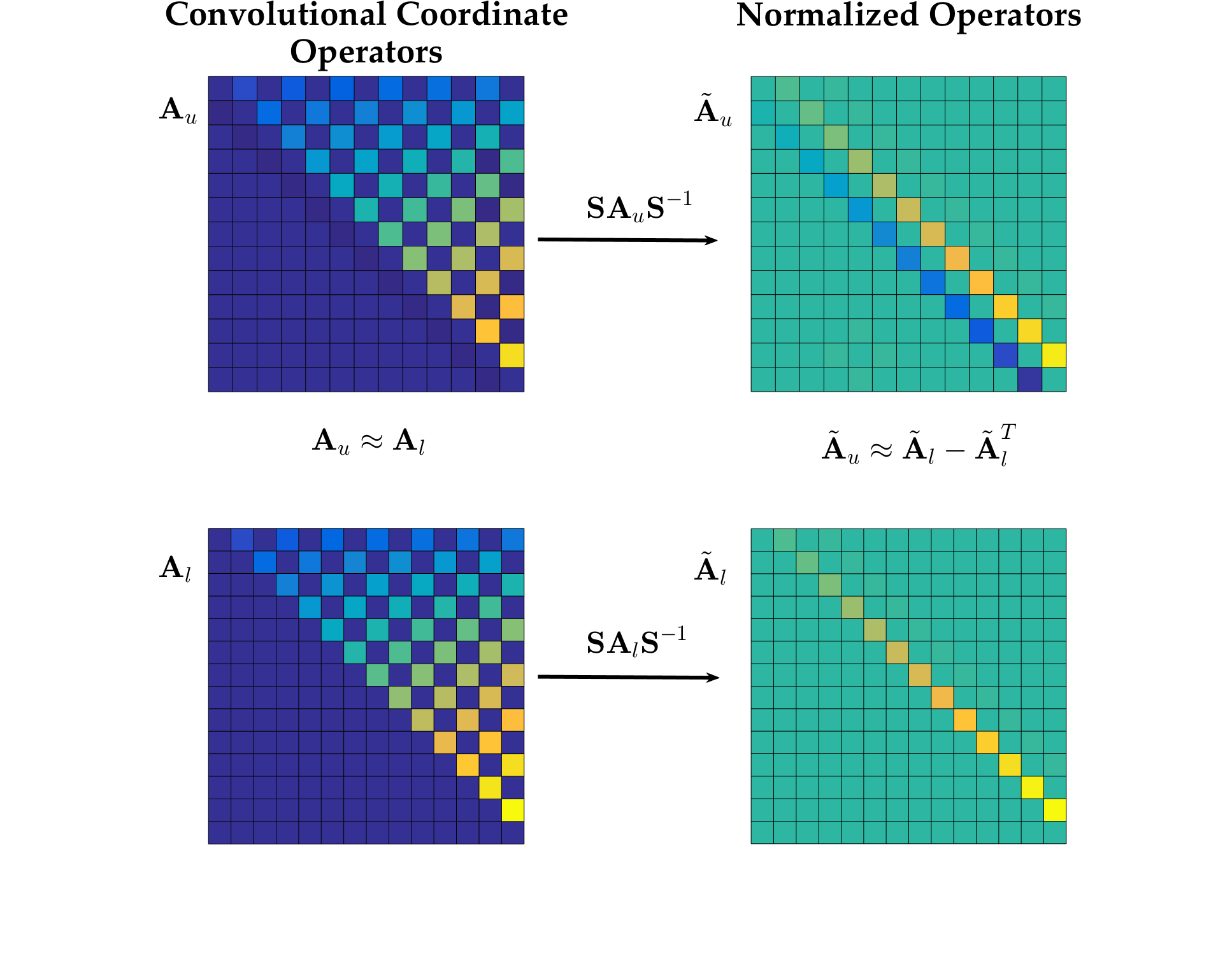}
    \vspace{-.3in}
    \caption{
    	Relationship between operators for the Lorenz system:
    	The linear model $\mathbf{A}_u$ is estimated in the SVD coordinates. The normalized linear model is given by $\tilde{\mathbf{A}}_u=\mathbf{S}^{-1}\mathbf{A}_u\mathbf{S}$. 
    	The SVD-coordinate model is similar to the predicted linear model $\mathbf{A}_l$ in Legendre coordinates.
    	In contrast, the normalized model $\tilde{\mathbf{A}}_l = \mathbf{S}^{-1}\mathbf{A}_l\mathbf{S}$ differs from $\tilde{\mathbf{A}}_u$, revealing how these operators encode very different dynamics.
    	The coefficients of these operators appear to grow linearly in $i$ and $j$. This connection explains the structure of the models derived in \cite{bruntonproctorkaiserkutz2017}.}
    \label{fig:matfig}
\end{figure}

\section{Summary and Discussion}

In this paper, we studied the properties of \textit{convolutional coordinates} on dynamical systems, constructed by convolving a basis of filters with the trajectories of the dynamical systems. We show that these coordinates naturally linearize system dynamics. For a given choice of basis used to the construct the convolutional coordinates, the representation of the Koopman operator in these coordinates is shown to be system independent. We derive these representations in terms of the basis functions and their derivatives. We consider the question of obtaining a good choice of basis for finite-dimensional approximations to the Koopman operator in convolutional coordinates. The eigenvectors of the autocorrelation function, or equivalently the singular vectors of the Hankel matrix, have several properties that make them attractive for this purpose, including that the analytically derived linear relations between these coordinates matches the approximations produced by applying DMD to these coordinates. We observe that these coordinates are also optimally parsimonious for finite discrete spectrum systems, and conjecture that restrictive error bounds exist for broader classes of systems. We validated our theoretical observations on a number of test systems, and found that these coordinates can excellently reconstruct the dynamics of these systems. We elaborate on the structure of the HAVOK models introduced by Brunton et. al. in  \cite{bruntonproctorkaiserkutz2017} by noticing a connection with the Legendre polynomials and SVD coordinates.

This work suggest many interesting directions for future research. While we have conjectured that SVD convolutional coordinates have good error bounds for a wide class of systems, more work is needed to fully characterize the quality of these approximations. Since these observables have a very precise structure and many restrictive properties, it is expected that this analysis will be more fruitful than past analyses of generic families of observable functions. Further study is also needed of the effects of various parameters on these approximations, including the embedding dimension and the state dimension. 
It is also expected that these observables will provide good quality Koopman approximations for a wide variety of systems, and we hope that these methods will be applied to the large range of dynamical systems where Koopman spectral information is desired. 
Because of the considerable promise of leveraging Koopman linear representations for the control of nonlinear systems~\cite{KordaMezic2018,Kaiser2017arxiv,PeitzKlus2019Auto}, it will be interesting to combine the delay coordinate representations with control approaches.

\section*{Acknowledgements} 
SLB and JNK acknowledge funding support from the Defense Advanced Research Projects Agency (DARPA-PA-18-01-FP-125).  SLB and EK acknowledge funding support from the Army Research Office (ARO W911NF-17-1-0306). 
EK gratefully acknowledges support by the ``Washington Research Foundation Fund for Innovation in Data-Intensive Discovery" and a Data Science Environments project award from the Gordon and Betty Moore Foundation (Award \#2013-10-29) and the Alfred P. Sloan Foundation (Award \#3835).
JNK acknowledges support from the Air Force Office of Scientific Research (AFOSR) grant FA9550-17-1-0329.

\begin{appendices}
\section{Delay Embedded Dynamics in Fourier and Legendre Bases}\label{ap:Legendre Basis}

We can derive the form of the coefficients $\mathbf{A}_{jk}$ for some example bases. The simplest such basis is the Fourier basis $\{e^{\pi int/\tau}\}$. The coefficients in this basis have the form
\begin{align*}
    \mathbf{A}_{jk} &= \frac{\pi i k}{\tau} \int_{-\tau}^\tau e^{\pi i (k - j) p / \tau}\,dp
    = 2 \pi i k \delta_{jk}.
\end{align*}

We can also consider these coefficients for a basis of orthogonal polynomials over $[-\tau,\tau]$. The foundation of our construction will be the Legendre polynomials, which are given by
\begin{equation}
    P_l(x) = \frac{1}{2^l} \sum_{k = 0}^{\text{floor}(l/2)} (-1)^k {l\choose k} {{2l - 2k}\choose l}x^{l-2k}.
\end{equation}
The Legendre polynomials have some useful properties. In particular, they are alternately even an odd, and the $n$th Legendre polynomial is orthogonal to the first $n - 1$ monomials $x^0, \cdots, x^{n-1}$.

Instead of working with these polynomials directly, we will work with a rescaled Legendre basis, that is orthonormal on $[-\tau,\tau]$. We define these functions as
\begin{align*}
    \phi_l(x) &= \frac{P_l(x/\tau)}{\sqrt{\tau}\Vert {P_l} \Vert} = \frac{P_l(x/\tau)}{\sqrt{\tau}}\sqrt{\frac{2l+1}{2}}\\
    &= \frac{1}{2^l}\sqrt{\frac{2l+1}{2\tau}} \sum_{k = 0}^{\text{floor}(l/2)} (-1)^k {l\choose k} {{2l - 2k}\choose l}\frac{x^{l-2k}}{\tau^{l-2k}}.
\end{align*}
For convenience we will abbreviate this expansion as
\begin{align*}
    \phi_l(x) = C_l \sum_{k = 0}^{\text{floor}(l/2)} B_{lk} x^{l - 2k}.
\end{align*}
The derivatives of these functions are given by
\begin{align*}
    \phi_l'(x)
    &= C_l \sum_{k = 0}^{\text{floor}((l-1)/2)} B_{lk} (l - 2k) x^{l - 2k - 1}.
\end{align*}
The coefficients are given by
\begin{align*}
    \mathbf{A}_{jk} &= \int_{-\tau}^\tau \phi_j(p)\bigg(C_k \sum_{n = 0}^{\text{floor}((k-1)/2)} (k - 2n) B_{kn}p^{k-2n-1} \bigg)\,dp.
\end{align*}
Note that the monomial exponents $k - 2n - 1$ reach at most $k - 1$. Since $\phi_j$ is orthogonal to the first $j - 1$ monomials, if $j \geq k$, $\mathbf{A}_{jk} = 0$.
For $j < k$, this simplifies to 
\begin{align*}
    \mathbf{A}_{jk} &= \int_{-\tau}^{\tau} \phi_j(p) \bigg(C_k \sum_{n = 0}^{\text{floor}((j + 1 - k)/2)} (k - 2n) B_{kn} p^{k - 2n - 1} \bigg)\,dp\\
    &= C_k \sum_{n = 0}^{\text{floor}((j + 1 - k)/2)} \bigg((k - 2n) B_kn \int_{-\tau}^\tau \phi_j(p) p^{k - 2n - 1}\,dp\bigg).
\end{align*}
The inner product between $\phi_j(p)$ and $p^{k-2n-1}$ is given by
\begin{align*}
    \int_{-\tau}^\tau \phi_j(p)p^{k - 2n - 1}\,dp &= \int_{-\tau}^\tau C_j \sum_{m = 0}^{\text{floor}(j/2)} B_{jm}p^{(j + k) - 2(m + n) - 1} \,dp\\
    &= C_j \sum_{m = 0}^{\text{floor}(j/2)} \frac{B_{jm}}{(j + k) - 2(m + n) - 1} p^{(j + k) - 2(m + n)} \bigg|_{-\tau}^\tau.
\end{align*}
Evaluating this gives
\begin{align*}
    \int_{-\tau}^\tau \phi_j(p)p^{k - 2n - 1}\,dp &= \begin{cases}
    2C_j \sum_{m = 0}^{\text{floor}(j/2)} \frac{B_{jm}}{(j + k) - 2(m + n) - 1} \tau^{(j + k) - 2(m + n)}&\text{($j + k)$ odd}\\
    0 &\text{$(j + k)$ even}.
    \end{cases}
\end{align*}
These can be substituted into our expression for $\mathbf{A}_{jk}$ to obtain the full Legendre coordinates.

\section{Fast Computation of Singular Vector Observables from the Autocorrelation Function}\label{ap:autocorrelation_approximation}

The naive approach to computing these observables from data requires computing the SVD of an $N \times m$ Hankel matrix, where $m$ is the number of snapshots and $N$ is the length of the delay embedding. The cost of this computation is $O(N^2m + m^2N + \min(N^3,m^3))$. In order to get effective approximations to smooth coordinates, we often need to take $N$ large. At first glance, it may seem like the apparent advantage of having a parsimonious basis of observables is mitigated by this large up front computational cost. 

To avoid this, we begin with the insight that we do not need to compute the full SVD, since we are only interested in computing the basis functions $\mathbf{U}$. These can be obtained by diagonalizing the autocovariance matrix:
\begin{equation}\label{eq:matrixautocorrelation}
\mathbf{A} = \mathbf{H}\mathbf{H}^T = \mathbf{U}\mathbf{S}^2\mathbf{U}^\dagger.
\end{equation}
Naively this procedure requires $N^2$ inner products of length-$m$ vectors to multiply to the two Hankel matrices. However, we can reduce this by using an analytic approximation of the autocovariance function. The autocovariance matrix is a discrete sampling of the autocovariance function, which is given by
\begin{equation}\label{eq:functionautocorrelation}
A(p,q) = \lim_{T\to\infty} \frac{1}{2T}\int_{-T}^{T} x(t+p)x(t+q)\,dt = \overline{x(t + p) x(t + q)}.
\end{equation}
The autocovariance is translationally invariant, i.e. $A(p,q) = A(p + t,q + t)$. We can therefore rewrite the autocorrelation as follows:
\begin{equation}
A(p,q) = \overline{x(t) x(t + (p - q))}.
\end{equation}
We can expand this in $(p-q)$ using a Taylor series approximation:
\begin{subequations}
	\begin{align}
	A(p,q) &= \lim_{T\to\infty} \frac{1}{2T}\int_{-T}^{T} x(t) \sum_{n = 0}^\infty \frac{x^{(n)}(t)(p-q)^n}{n!}\,dt\\
	&= \sum_{n=0}^\infty \frac{(p - q)^n}{n!} \overline{x(t) x^{(n)}(t)}. \label{eq:autocorrelTaylor}
	\end{align}
\end{subequations}
If we truncate this Taylor series at some value $n = n_{max}$, we only need to compute $n_{max}$ inner products of length $m$ instead of $N^2$. This represents a significant cost savings over the naive SVD

This approach generalizes when the signal $\mathbf{x}$ is multivariate. In this case, the autocorrelation is a tensor function:
\begin{equation}\label{eq:multivariateautocorrelation}
\mathbf{A}(p,q) = \overline{\mathbf{x}(t+p)\mathbf{x}^*(t+q)}.
\end{equation}
We can obtain a similar Taylor expansion for each component of the autocorrelation separately:
\begin{equation}
A_{jk}(p,q) = \sum_{n=0}^\infty \frac{(p - q)^n}{n!} \overline{ x_j(t) x_k^{(n)}(t)}.
\end{equation}
For a state with dimension $k$, the number of inner products in this approach is $k^2 n_{max}$ as opposed to a naive $k^2 N^2$.

\end{appendices}

\bibliographystyle{plain}
\bibliography{references}

\end{document}